\documentclass{article}
\usepackage{amssymb,amsmath}
\usepackage[pdftex]{graphicx}
\usepackage{ dsfont }
\usepackage{lscape}

\def\qed{\hfill {\hbox{${\vcenter{\vbox{               
   \hrule height 0.4pt\hbox{\vrule width 0.4pt height 6pt
   \kern5pt\vrule width 0.4pt}\hrule height 0.4pt}}}$}}}

\newtheorem{theorem}{Theorem}
\newtheorem{definition}{Definition}

\newtheorem{proposition}[theorem]{Proposition}

\newtheorem{example}{Example}
\newtheorem{remark}[example]{Remark}

\newtheorem{conjecture}{Conjecture}

\newenvironment{proof}[1][Proof]{\smallskip\noindent{\bf #1.}\quad}%
{\qed\par\medskip}

\title{\Large \textbf{On the Coloring of Pseudoknots}}

\author{
Allison Henrich
 \footnote{henricha@seattleu.edu, Seattle University, Seattle, WA 98122, United States}\hspace{1cm}
Slavik Jablan\footnote{sjablan@gmail.com, The Mathematical Institute, Belgrade, 11000, Serbia}}

\begin{document}

\maketitle

\begin{abstract}
Pseudodiagrams are diagrams of knots where some information about which strand goes over/under at certain crossings may be missing. Pseudoknots are equivalence classes of pseudodiagrams, with equivalence defined by a class of Reidemeister-type moves. In this paper, we introduce two natural extensions of classical knot colorability to this broader class of knot-like objects. We use these definitions to define the determinant of a pseudoknot (i.e. the pseudodeterminant) that agrees with the classical determinant for classical knots. Moreover, we extend Conway notation to pseudoknots to facilitate the investigation of families of pseudoknots and links. The general formulae for pseudodeterminants of pseudoknot families may then be used as a criterion for $p$-colorability of pseudoknots.
\end{abstract}


\section{Introduction}
\subsection{Pseudodiagrams and pseudoknots}

Recently, Ryo Hanaki introduced the notion of a {\em pseudodiagram} of a knot, link, and spatial graph~\cite{hanaki}. A pseudodiagram of a knot or link is a knot or link diagram that may be missing some crossing information, as in Figure~\ref{f0}. In other words, at some crossings in a pseudodiagram, it is unknown which strand passes over and which passes under. These undetermined crossings are called {\em precrossings} and are pictured simply as self-intersections. Special classes of pseudodiagrams are knot diagrams and knot {\em shadows}, i.e. pseudodiagrams containing {\em only} precrossings. Pseudodiagrams were originally considered because of their potential to serve as useful models for biological objects related to DNA, but they are interesting objects in their own right. Pseudodiagrams and their virtual counterparts have also been studied in~\cite{SMALL}.

\begin{figure}[th]
\centerline{\includegraphics[width=3.8in]{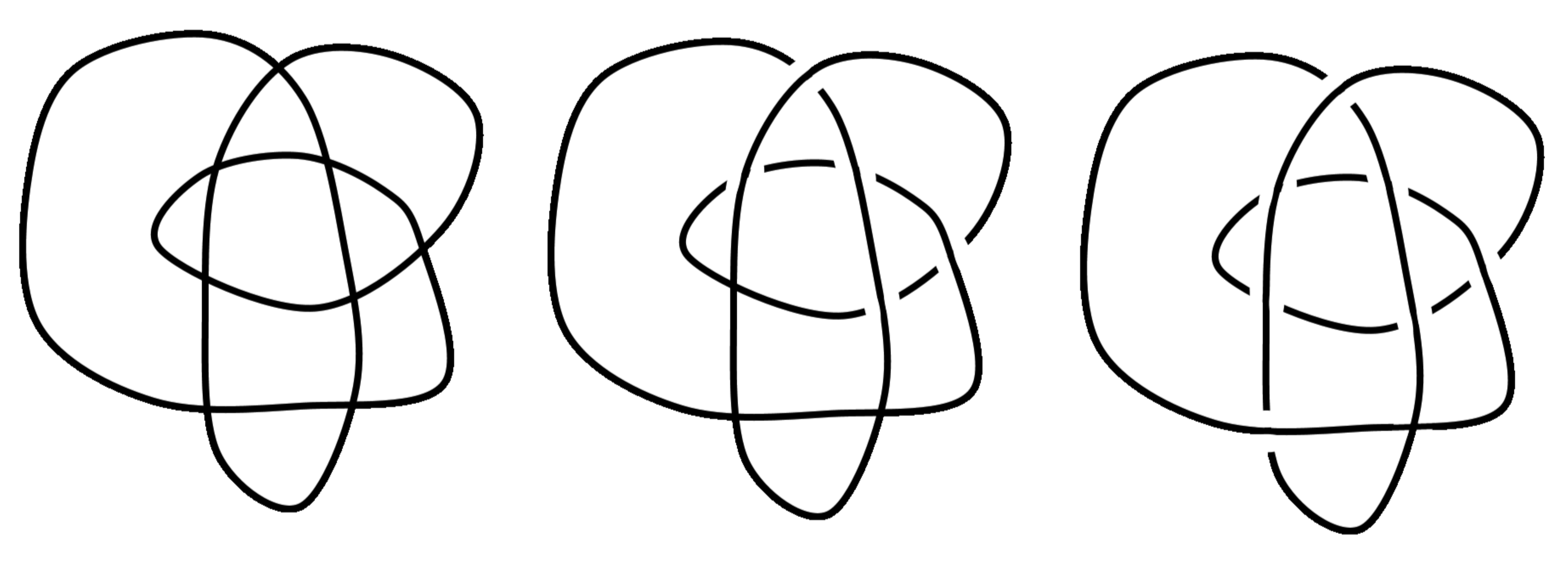}} \vspace*{8pt}
\caption{Examples of pseudodiagrams. \label{f0}}
\end{figure}

In~\cite{ps}, the idea of pseudodiagrams is extended to {\em pseudoknots} (and {\em pseudolinks}), i.e. equivalence classes of pseudodiagrams modulo pseudo-Reidemeister moves (shown in Figure~\ref{rmoves}). The primary invariant explored in~\cite{ps} is the {\em WeRe-set} (Weighted Resolution set) of a pseudoknot. The WeRe-set of a pseudoknot is the set of ordered pairs where the first entry of a pair is a knot that may be realized by resolving all crossings in a pseudodiagram of the pseudoknot, and the second entry is the probability of obtaining the given knot type in a random resolution of precrossings.
While the WeRe-set is a powerful invariant, it may be difficult to compute for pseudoknots with large precrossing numbers. Thus, our aim is to expand the number of tools that can be used for classification of pseudoknots. One natural invariant to consider, pseudoknot colorability, is the subject of this paper.

\begin{figure}[h]
\centerline{\includegraphics[width=4.3in]{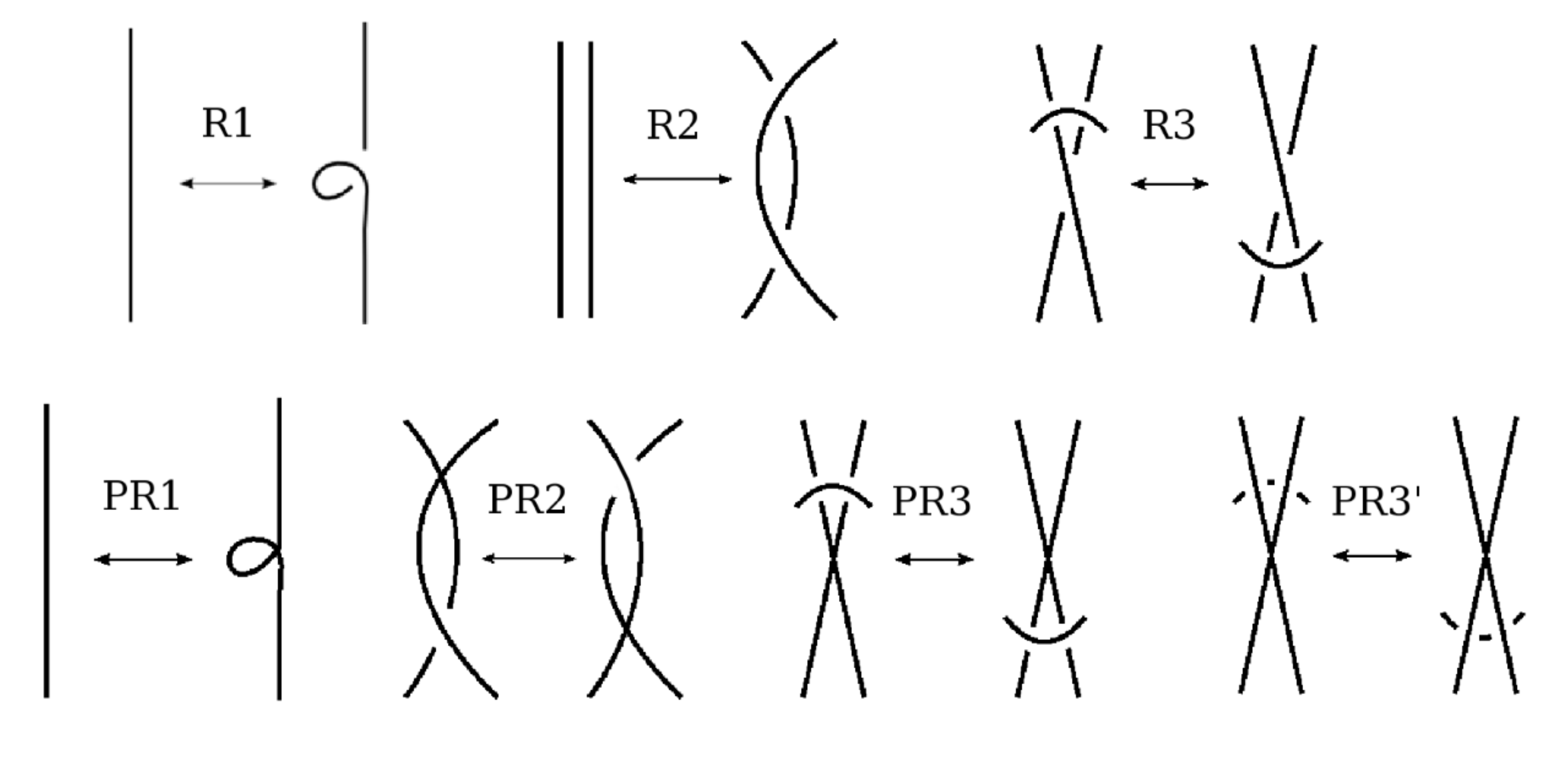}} \vspace*{8pt}
\caption{The pseudo-Reidemeister moves. \label{rmoves}}
\end{figure}

\subsection{Knot colorability}

For knots, a $p$-{\em coloring} of a diagram is an assignment of elements of $\mathds{Z} /p \mathds{Z}$ to arcs in the diagram such that at each crossing, two times the element assigned to the over-arc equals the sum of the elements assigned to the under-arcs. It can easily be shown that if one diagram of a knot has a non-trivial $p$-coloring, then so does every other diagram of the knot. We say that a knot is $p$-{\em colorable} (or {\em colorable mod} $p$) if there exists a non-trivial $p$-coloring of each of its diagrams.

If $p=3$, we note that the coloring condition for the crossings is satisfied if and only if the three arcs that meet at a crossing are either colored with all different ``colors'' (that is, elements of $\mathds{Z} /3 \mathds{Z}$) or all the same color.

In general, to determine for which $p$ a given knot is $p$-colorable, we begin by fixing a diagram $D$ of the knot. In $D$, we assign a variable to each arc and write an equation of the form $x+y-2z=0$ for every crossing, where $x$ and $y$ represent the variables assigned to the under-arcs and $z$ represents the variable assigned to the over-arc. We obtain in this way a system of $n$ equations in $n$ variables, where $n$ is the crossing number of the diagram (which is also equal to the number of arcs in the diagram). Forming a matrix corresponding to this system, we say the {\em determinant} of the knot is the determinant of any $n-1$ submatrix (which is independent of the particular submatrix chosen). The knot given by diagram $D$, then, is $p$-colorable if and only if $p$ divides this determinant.

\subsection{Pseudoknot colorability}
There are two extensions of the classical notion of $p$-colorability we'd like to explore in this paper. One notion we refer to as strong $p$-colorability, and a weaker notion we will simply refer to as $p$-colorability. First, let us be clear that arcs in a pseduodiagram begin and end at classical crossings only. In other words, arcs may merely pass through precrossings. 

\begin{definition}\label{tri} A pseudodiagram is {\em strong $p$-colorable} (or {\em strong colorable mod $p$}) if the arcs of the diagram can be ``colored" (i.e. labelled) with elements of $\mathds{Z}/p\mathds{Z}$, subject to the following conditions.
\begin{enumerate}
\item Given a classical crossing, two times the element assigned to the over-arc is equal to the sum of the elements assigned to the under-arcs, mod $p$.
\item For each precrossing $P$, there is a neighborhood of $P$ and an element $a\in\mathds{Z}/p\mathds{Z}$ such that all arcs of the diagram in the neighborhood are colored with $a$.
\end{enumerate}
\end{definition}

In Section~\ref{pk_color}, we show that this notion of colorability is a pseudoknot invariant.  We also explore the following notion of colorability. Two pseudoknots are said to be {\em $p$-colorable} (or {\em colorable mod $p$}) if all of its resolutions are $p$-colorable classical knots. This, too is an invariant. It is not difficult to see that if a pseudoknot is strong $p$-colorable, then it is $p$-colorable. We use $p$-colorability for pseudoknots to define the {\em pseudodeterminant} of a pseudoknot.

Before we discuss colorability further in Section~\ref{pk_color}, we begin the body of our paper in Section~\ref{conway_bg} by providing background on Conway notation for classical knots and extending this notation to pseudoknots. This will facilitate our discussion of the colorability of families of pseudoknots in Section~\ref{example_conj}.

\section{Conway notation for classical and pseudoknots}\label{conway_bg}

\subsection{Classical Conway notation}
All classical knots and links (or shortly $KL$s) can be described using Conway notation. For readers unfamiliar with the conventions, we describe Conway notation as it was introduced in Conway's seminal paper
\cite{1} published in 1967, and effectively used since (e.g.,
\cite{7a,8a,9a}). Conway symbols of knots with up to $10$ crossings and
links with at most $9$ crossings are given in the Appendix of the
book \cite{7a}.

\begin{figure}[th]
\centerline{\includegraphics[width=2in]{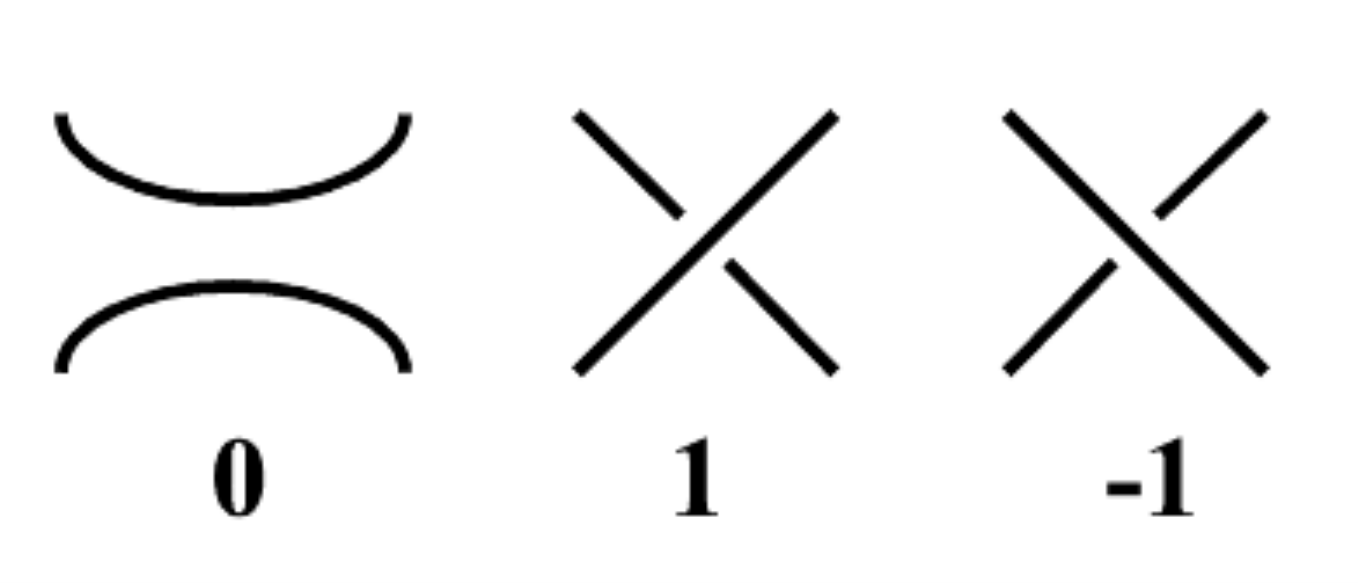}} \vspace*{8pt}
\caption{The elementary tangles. \label{con1}}
\end{figure}

The main building blocks in the Conway notation are elementary
tangles. We distinguish three elementary tangles,  shown in Fig.
\ref{con1}. These are commonly denoted by 0, 1, and $-1$. All other tangles can be
obtained by combining elementary tangles, while 0 and 1 are
sufficient for generating alternating $KL$s.
Elementary tangles can be combined by the following three operations: {\em
sum}, {\em product}, and {\em ramification} (Figs.
\ref{con2}-\ref{con3}). Given tangles $a$ and $b$, we denote by $-a$ the image of $a$ under
reflection about the line joining the Northwest and Southeast corners of a the tangle diagram. The sum of $a$ and $b$ is denoted by $a+b$, while the product $a\,b$ is defined as $a\,b = -a+b$. Furthermore, the binary operation {\em ramification} is given by $(a,b) = -a-b$.

\begin{figure}[th]
\centerline{\includegraphics[width=3.8in]{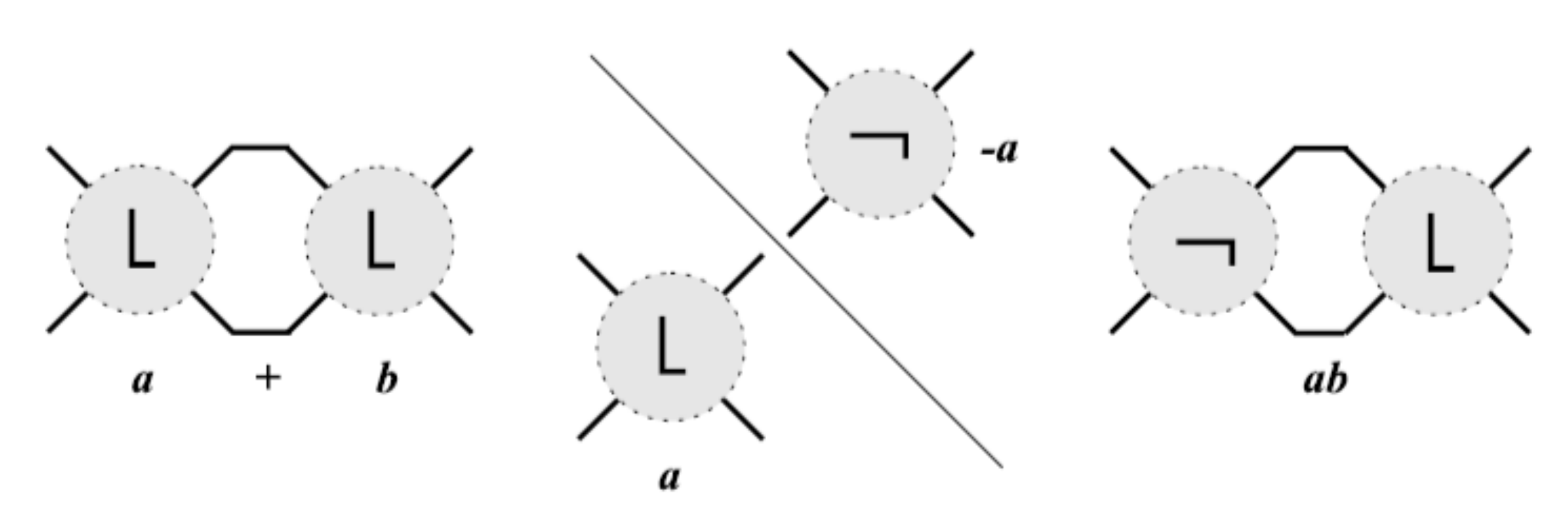}} \vspace*{8pt}
\caption{A sum and product of tangles. \label{con2}}
\end{figure}

\begin{figure}[th]
\centerline{\includegraphics[width=1.4in]{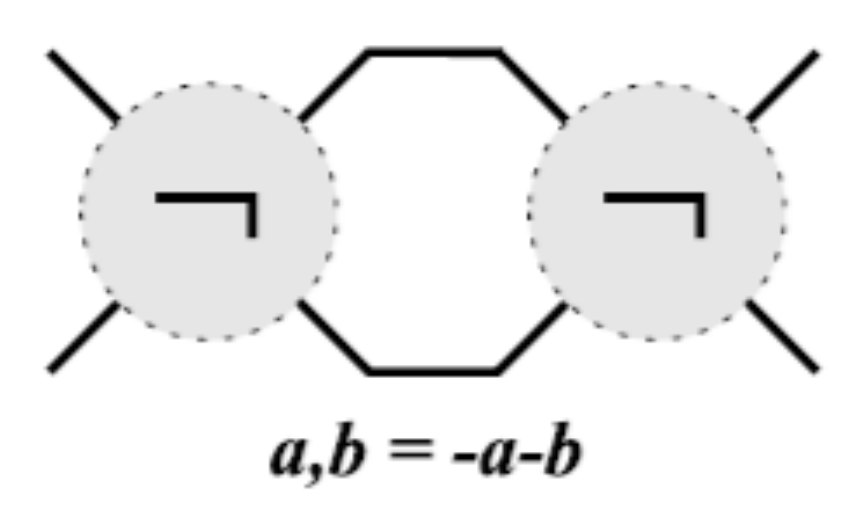}} \vspace*{8pt}
\caption{Ramification of tangles. \label{con3}}
\end{figure}

A tangle can be closed in two ways without
introducing additional crossings. We can join the pairs of ends at the Northeast and Northwest corners of the tangle diagram as well as the
Southeast and Southwest ends of the tangle to obtain the {\em numerator closure}. The closure obtained by joining the Northeast and Southeast ends as well as the Northwest and Southwest ends is called the {\em
denominator closure} (Fig. \ref{con5}a,b).

\begin{figure}[th]
\centerline{\includegraphics[width=2.4in]{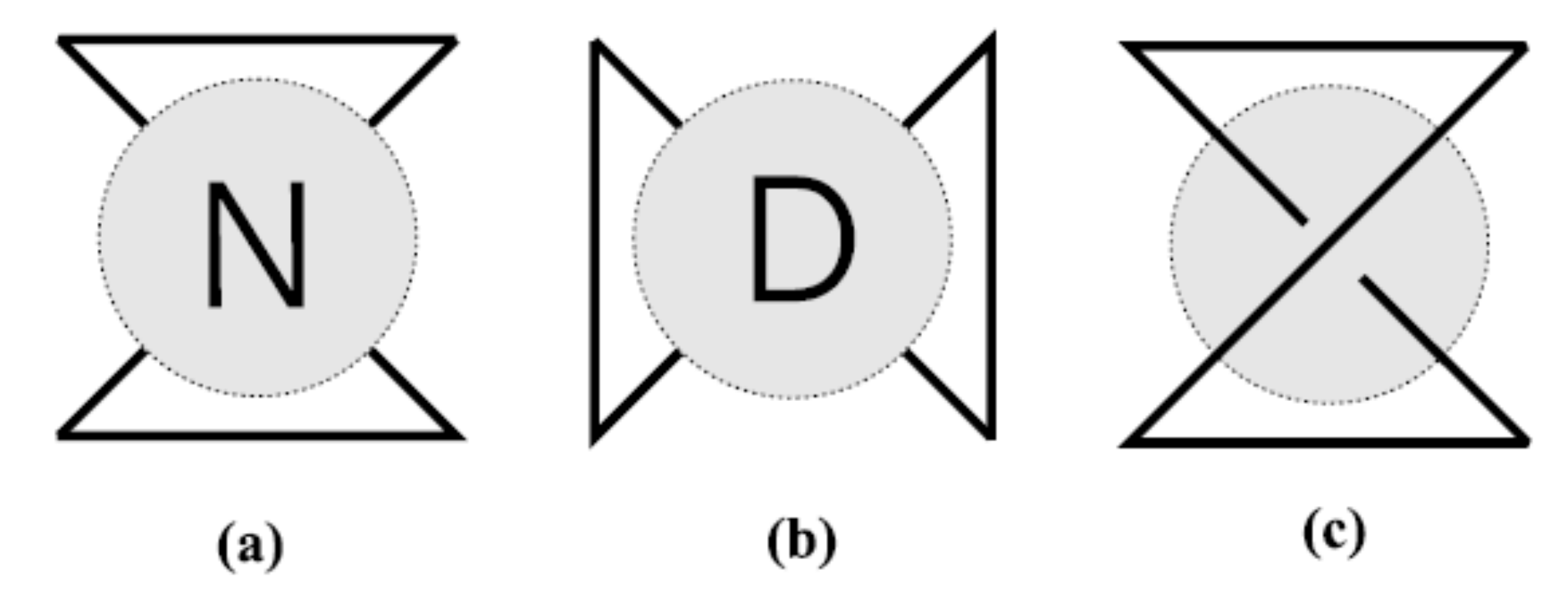}} \vspace*{8pt}
\caption{(a) Numerator closure; (b) denominator closure; (c) basic
polyhedron $1^*$. \label{con5}}
\end{figure}

\begin{definition}
A {\em rational tangle} is any tangle that is equivalent to a finite product of elementary tangles.
A {\em rational $KL$} is a knot or link that can be obtained by taking the numerator closure of a rational tangle.
\end{definition}

\begin{definition}
A tangle is {\em algebraic} if it can be obtained from elementary
tangles using the operations of sum and product. An {\em
algebraic} $KL$ is a knot or link that can be obtained by taking the numerator closure of an algebraic tangle.
\end{definition}

\begin{example} A {\em Montesinos tangle} and the corresponding {\em Montesinos
$KL$} consist of $n$ alternating rational tangles $t_i$, including at
least three non-elementary tangles $t_k$ for $k\in
\{1,2,\ldots,n\}$. (See Fig. \ref{con6}.) We denote a Montesinos link by $t_1,t_2,\ldots,t_n$, where $n\ge 3$. The number of tangles, $n$, is called the
{\em length of the Montesinos tangle}.

If all tangles $t_i$ for $i=1,...,n$ are {\em integer tangles} (i.e. tangles of the form $1+1+\cdots 1$ or $-1-1\cdots-1$), where $n\ge 3$, we obtain {\em pretzel}
$KL$s.\end{example}

\begin{figure}[th]
\centerline{\includegraphics[width=1.4in]{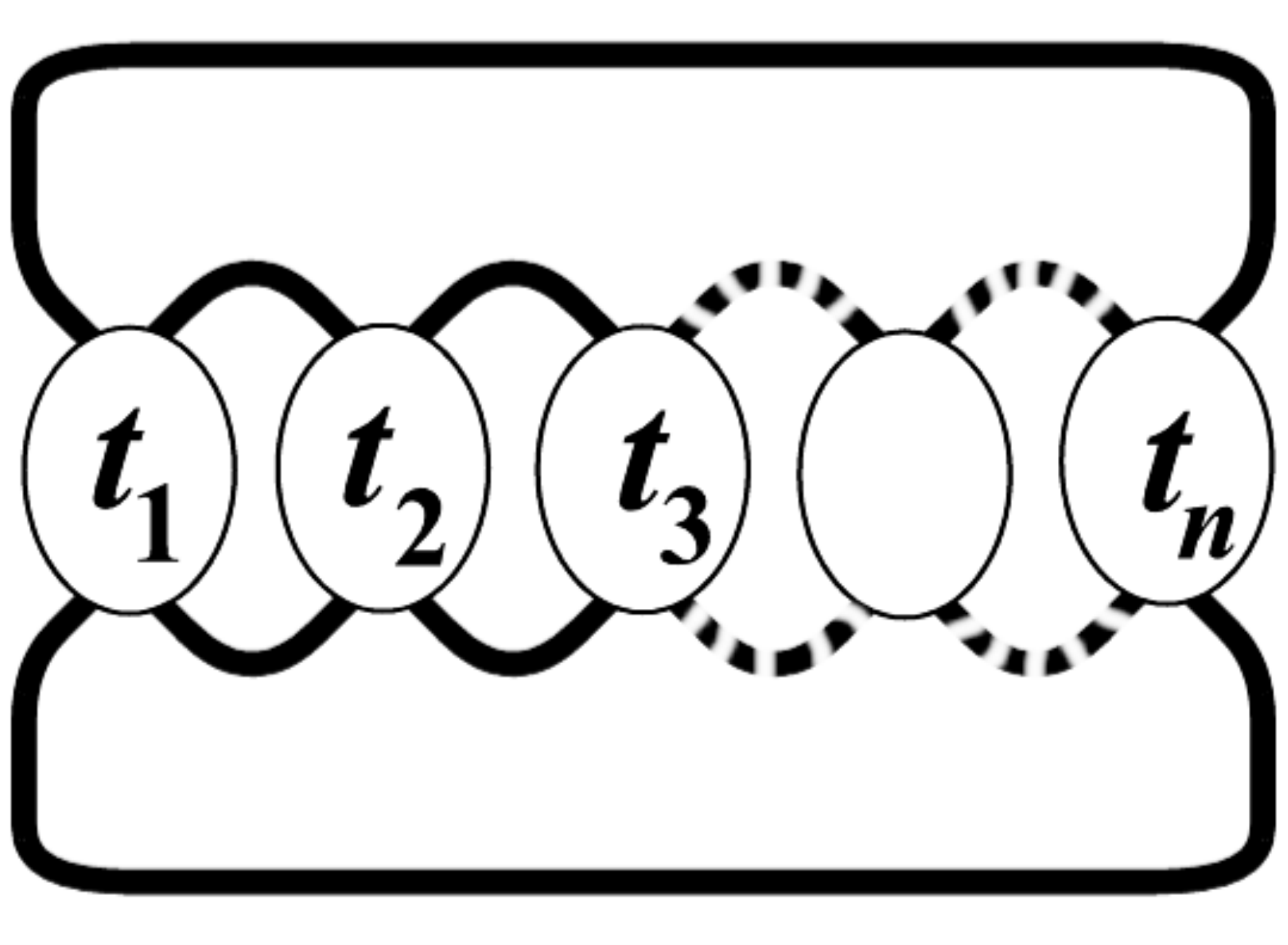}} \vspace*{8pt}
\caption{Montesinos link $t_1,t_2,\ldots,t_n$. \label{con6}}
\end{figure}

\begin{definition}
A {\em basic polyhedron} is a $4$-regular, $4$-edge-connected, at least $2$-vertex connected plane graph.
\end{definition}

A basic polyhedron \cite{1,7a,8a,9a} of a given $KL$ can be identified by
recursively collapsing all bigons in a $KL$ diagram, until no bigons remain \ref{collapse}.

\begin{figure}[th]
\centerline{\includegraphics[width=4.4in]{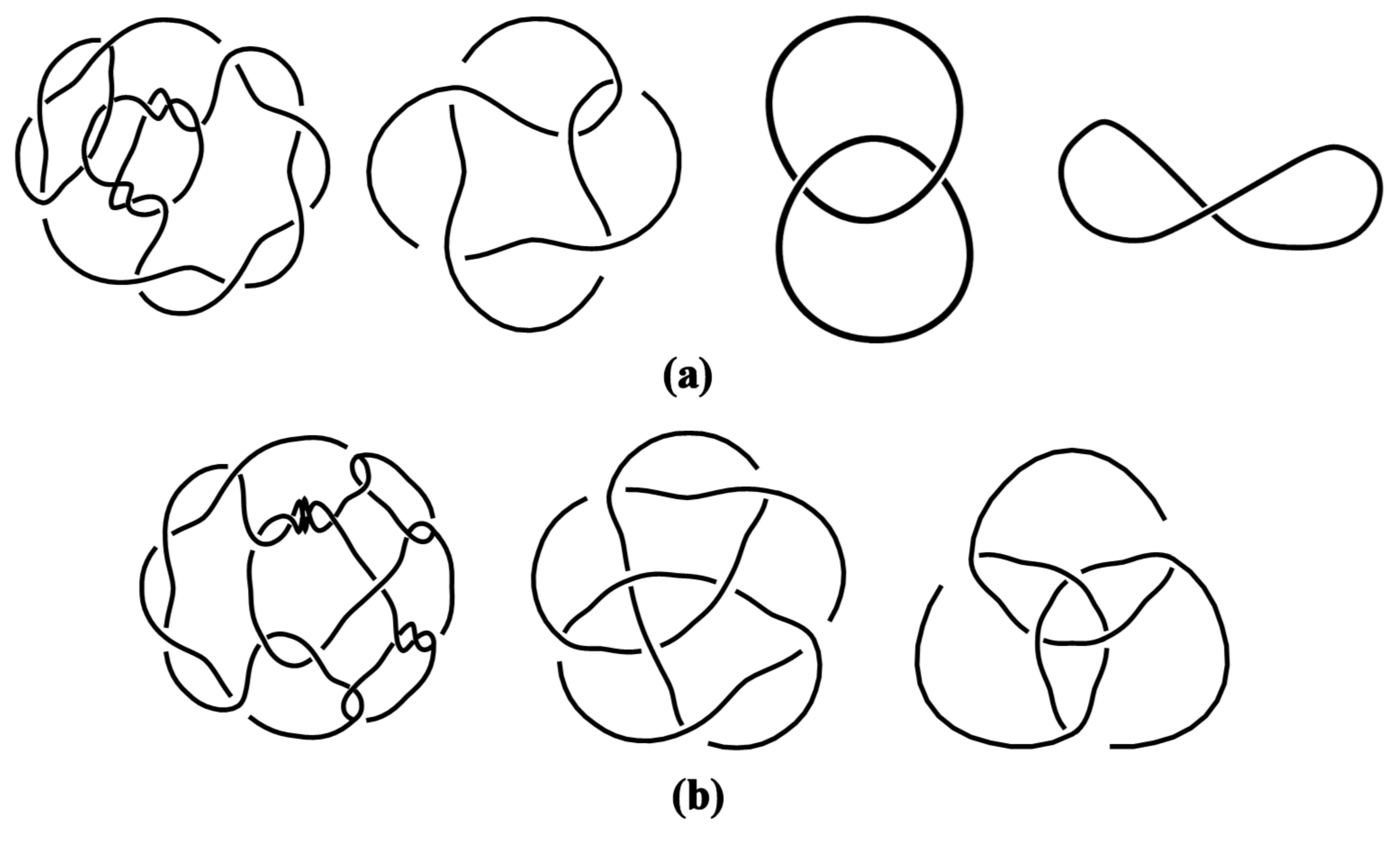}} \vspace*{8pt}
\caption{(a) Bigon collapse in the link $(5,3,2)\,(3,3)\rightarrow (1,1,1)\,(1,1)=3\,2 \rightarrow 1^*$; (b) bigon collapse in the knot $6^*(3,2).2.4.5.(2,2) \rightarrow 6^*(1,1).1.1.1.(1,1)=6^*2::2 \rightarrow 6^*$. \label{collapse}}
\end{figure}

The basic polyhedron $1^*$ is illustrated in Fig. \ref{con5}c, and the
basic polyhedron $6^*$ is illustrated in Fig.
\ref{con8}.

\begin{definition}
A link $L$ is {\em algebraic} or a $1^*$-{\em link} if there
exists at least one diagram of $L$ which can be reduced to the basic
polyhedron $1^*$ by a finite sequence of bigon collapses. Otherwise,
it is a {\em non-algebraic} or {\em polyhedral link}.
\end{definition}

Conway notation for polyhedral $KL$s begins with the symbol
denoting a basic polyhedron (together with a standard ordering of its vertices). The symbol
$n^{*m}=n^{*m}1.1.\ldots.1$ (where $*m$ is the shorthand for a sequence of $m$ stars)
denotes the $m$-th basic polyhedron in the list of basic polyhedra
with $n$ vertices. A $KL$ obtained from the basic polyhedron $n^{*m}$
by substituting the first $k$ vertices with tangles $t_1$, $\ldots$, $t_k$, where $k \leq n$ is denoted by $n^{*m}t_1.t_2.\ldots .t_k$. A common shorthand convention allows us to replace a substituent of value 1 by a colon. If $k<n$, we replace each of the vertices $k+1, ..., n$ by the elementary tangle 1. Let us illustrate how this notation works with some examples. The notation $6^*2:2:2\,0$ is shorthand for
$6^*2.1.2.1.2\,0.1$, and $6^*2\,1.2.3\,2:-2\,2\,0$ is shorthand for
$6^*2\,1.2.3\,2.1.-2\,2\,0.1$ Both examples are pictured in Fig. \ref{con8}.

\begin{figure}[th]
\centerline{\includegraphics[width=4.8in]{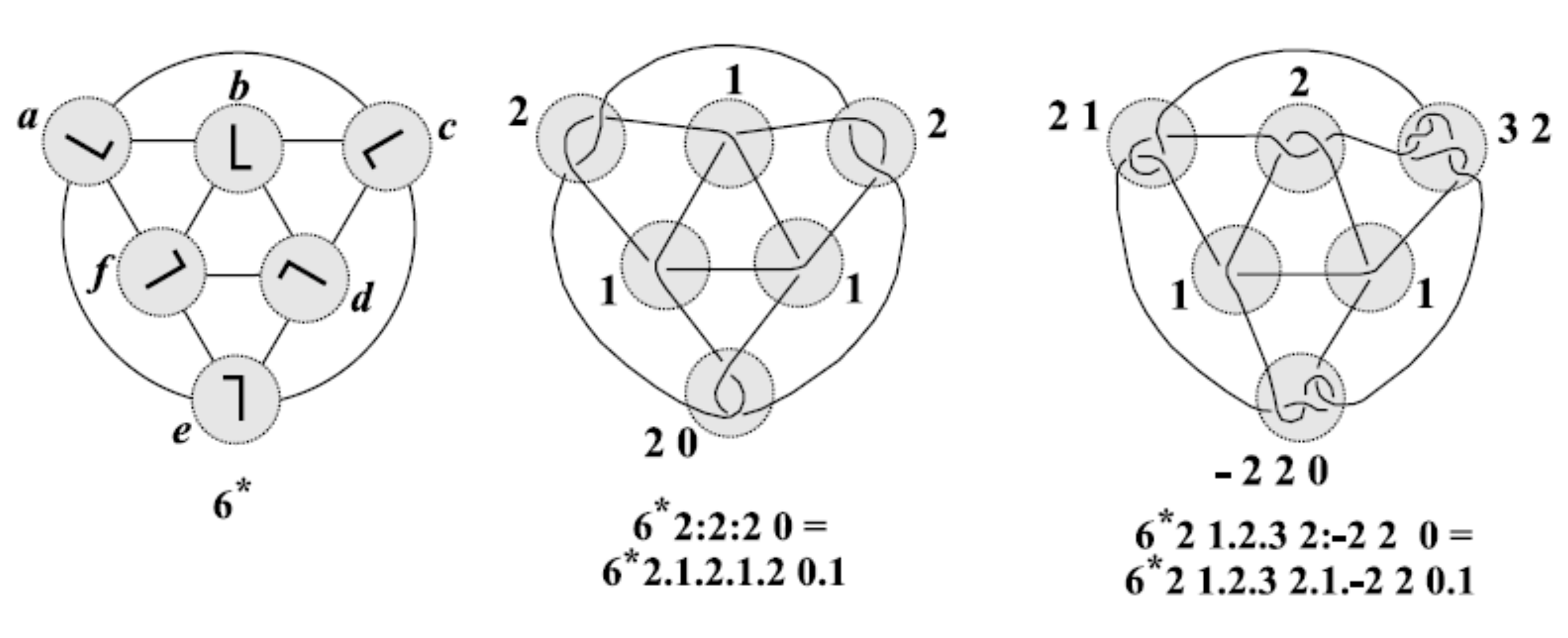}} \vspace*{8pt}
\caption{Basic polyhedron $6^*$ and the knots $6^*2.1.2.1.2\,0.1$
and $6^*2\,1.2.3\,2:-2\,2\,0$. \label{con8}}
\end{figure}

We now recall a useful technique for defining a family of $KL$s from a particular knot or link.

\begin{definition}
For a knot or link, $L$, given using unreduced\footnote{The Conway
notation is called {\em unreduced} if the 1's that denote replacing vertices by the elementary tangle 1 are not omitted. E.g. $6^*2.1.2.1.2\,0.1$ rather than $6^*2:2:2\,0$.} Conway
notation $C(L)$, denote by $S$ the set of numbers in the Conway symbol,
where numbers denoting basic polyhedra and zeros (determining
the position of tangles in the vertices of polyhedron) are excluded. Let
$\tilde S=\{a_1,a_2, \ldots, a_k\}$ be a non-empty subset of $S$.
The family $F_{\tilde S}(L)$ of $KL$s derived from $L$ consists
of all knots or links $L'$ whose Conway symbol is obtained by
substituting all $ a_i\neq \pm 1$, by $sgn(a_i) |a_i+k_{a_i}|$,
$|a_i+k_{a_i}| >1$, where $k_{a_i} \in Z$.
\end{definition}

An infinite subset of a family of $KL$s is called {\em subfamily}. Note that if all
$k_{a_i}$ are even integers, the number of components is preserved
within the corresponding subfamilies, i.e., adding full-twists
preserves the number of components inside the subfamilies.

\begin{definition}
A link given by a Conway symbol containing only tangles $\pm 1$ and  $\pm 2$
is called a {\em source link}. A link given by a Conway
symbol containing only tangles $\pm 1$, $\pm 2$, or $ \pm 3$ is
called a {\em generating link}.
\end{definition}

For example, the Hopf link, $2$, (which is link $2_1^2$ in Rolfsen's notation) is
the source link of the simplest link family, $p$, ($p=2,3,\ldots $)
shown in Fig. \ref{con12}. The Hopf link and the trefoil, $3$, (knot $3_1$ in the
classical notation) are generating links of this family. A family of
$KL$s is usually derived from its source link by substituting $a_i
\in \tilde S$, where $a_i=\pm 2$, by $sgn(a_i) (2+k)$, $k=1,2,3,\ldots $
(see Def. 5 and Def. 6).

\begin{figure}[th]
\centerline{\includegraphics[width=3in]{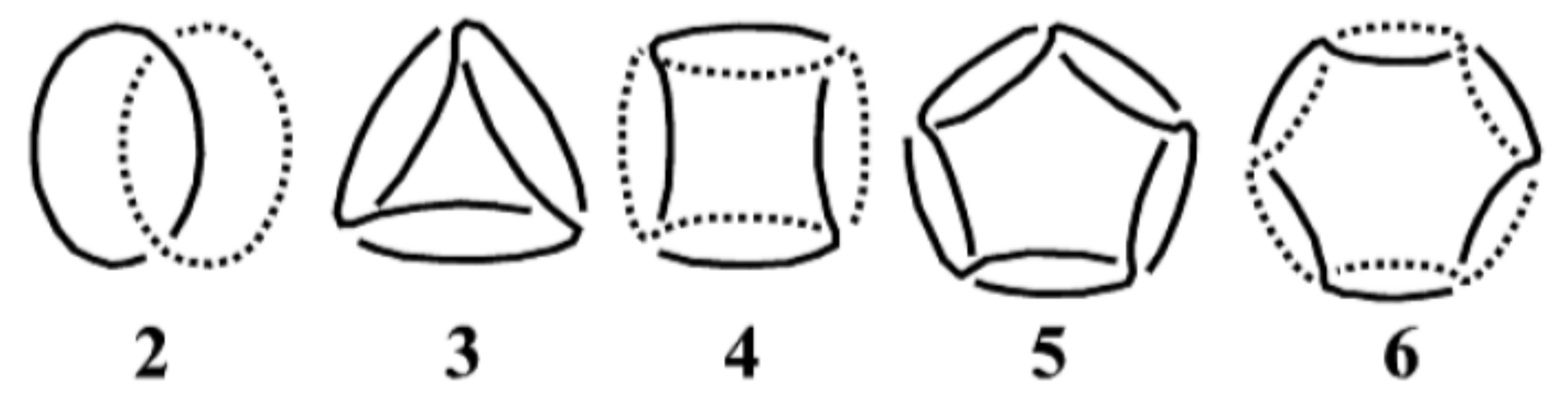}} \vspace*{8pt}
\caption{Hopf link $2_1^2=2$ and its family $p$ ($p=2,3,\ldots $).
\label{con12}}
\end{figure}

\subsection{Conway notation for pseudoknots and links}

Analogous to the Conway notation for classical $KL$s
we use extended Conway notation for pseudodiagrams and pseudolinks, by adding to the list of the elementary tangles $0$, $1$, $-1$ the elementary tangle $i$ denoting a precrossing. To be clear, the extended
Conway notation for pseudoknots differs from the standard one in
the following way:
\begin{itemize}
  \item pseudocrossings are denoted by $i$.
  \item sequence of $n$ classical crossings $1,\ldots ,1$ (positive $n$-twist) is denoted by $1^n$
  \item sequence of $n$ classical negative crossings  $-1,\ldots ,-1$ (negative $n$-twist) is denoted by
$(-1)^n$
\end{itemize}

The convention introduced in the Conway notation extended to pseudodiagrams is natural, given that every positive $n$-twist can be denoted by $1,\ldots ,1=1^n$, and every negative $n$-twist by $-1,\ldots ,-1.$ This convention extends to pseudoknot families. For example, consider the simplest $KL$ family $2_1^2$, $3_1$, $4_1^2$, $5_1$, $\ldots$ of knots and links, generated by the Hopf link, $2$, and the trefoil, $3$, where the family member with $p$ crossings is denoted by $p$, as described above. We generate a family of pseudoknots and links by substituting a crossing in each member of the family by a precrossing, we obtain the family of pseudoknots and pseudolinks $(i,1)$, $(i,1^2)$, $(i,1^3)$, $(i,1^4)$, $\ldots$, and in general $(i,1^{p-1})$  for $p\ge 2$ (Fig. \ref{pseudofam}).

\begin{figure}[th]
\centerline{\includegraphics[width=4in]{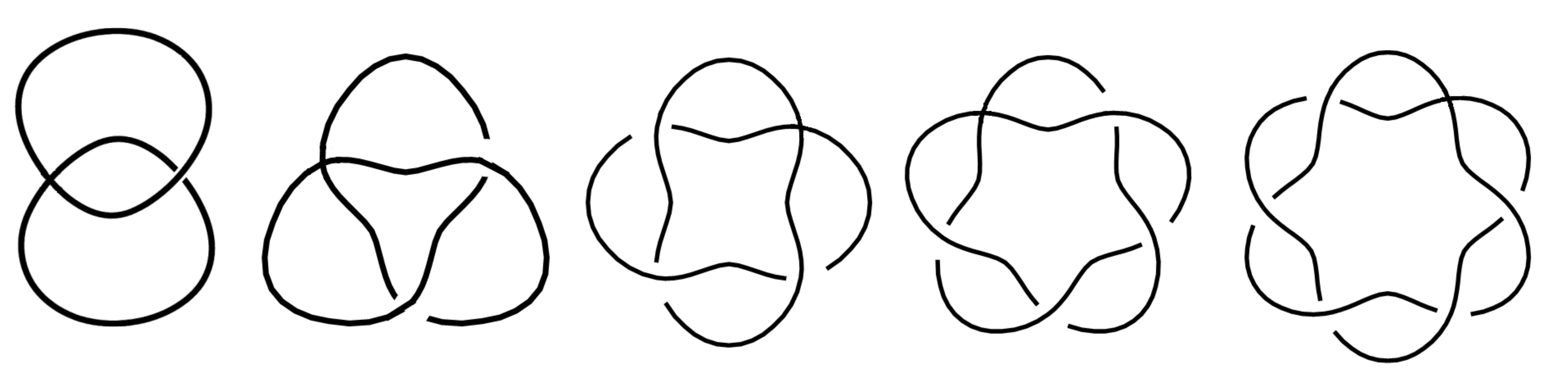}} \vspace*{8pt}
\caption{Family of pseudoknots and pseudolinks   $i,1^{p-1}$ ($p\ge 2$).
\label{pseudofam}}
\end{figure}

\section{Colorability of pseudoknots}\label{pk_color}

We now return to our study of coloring invariants. In our introduction, we defined two notions of colorability for pseudoknots: strong $p$-colorability and $p$-colorability. We prove that these notions both yield pseudoknot invariants.

\begin{theorem}
Strong $p$-colorability is an invariant of pseudoknots.
\end{theorem}

\begin{proof}
Because $p$-colorability is an invariant of classical knots and strong $p$-colorability for {\em classical} pseudoknots is equivalent to ordinary $p$-colorability for classical knots, we can be sure that the classical Reidemeister moves preserve strong $p$-colorability. Let us, then, consider the pseudo-Reidemeister (PR) moves.

First we note that the PR1 move preserves strong $p$-colorability, for the arcs involved in the move must be colored by the same element. This is simply due to the fact that arcs cannot begin and end at a precrossing.

Considering the PR2 move, all but one arc involved in the move is forced by the precrossing condition to be colored with a single element. Thus, the remaining arc must be colored with the same element (since $a$ is the only solution to $2a=a+x  (mod\,p)$), regardless of whether we are considering the local picture before or after the move has been performed. See Figure~\ref{PRfig}.

The PR3 and PR$3'$ moves are somewhat more complicated. Consider the PR3 move where the free strand passes under the precrossing. The arcs passing through the precrossing must be the colored with the same element (let's call it $a$), but the remaining arcs may be colored differently. Consider the central arc (i.e. the one labeled with a $b$ in the bottom left corner of Figure~\ref{PRfig}). Regardless of whether or not $b$ differs from $a$, the remaining two arcs must be colored with $c=2a-b (mod p)$ (in particular, $c=a$ in the case that $b=a$). This coloring uniquely determines the valid coloring of the local picture after the PR3 move has been performed.

\begin{figure}[th]
\centerline{\includegraphics[height=1.5in]{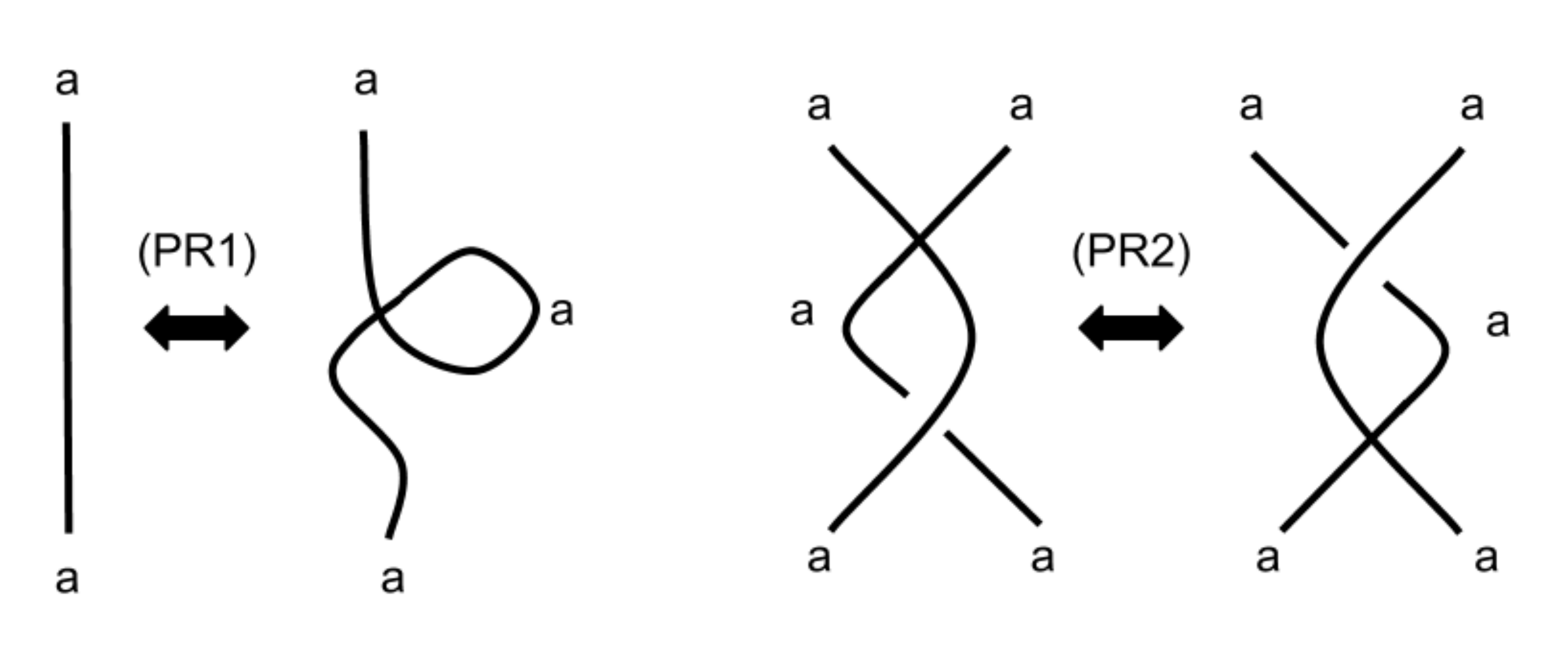}} \vspace*{8pt}
\centerline{\includegraphics[height=1.5in]{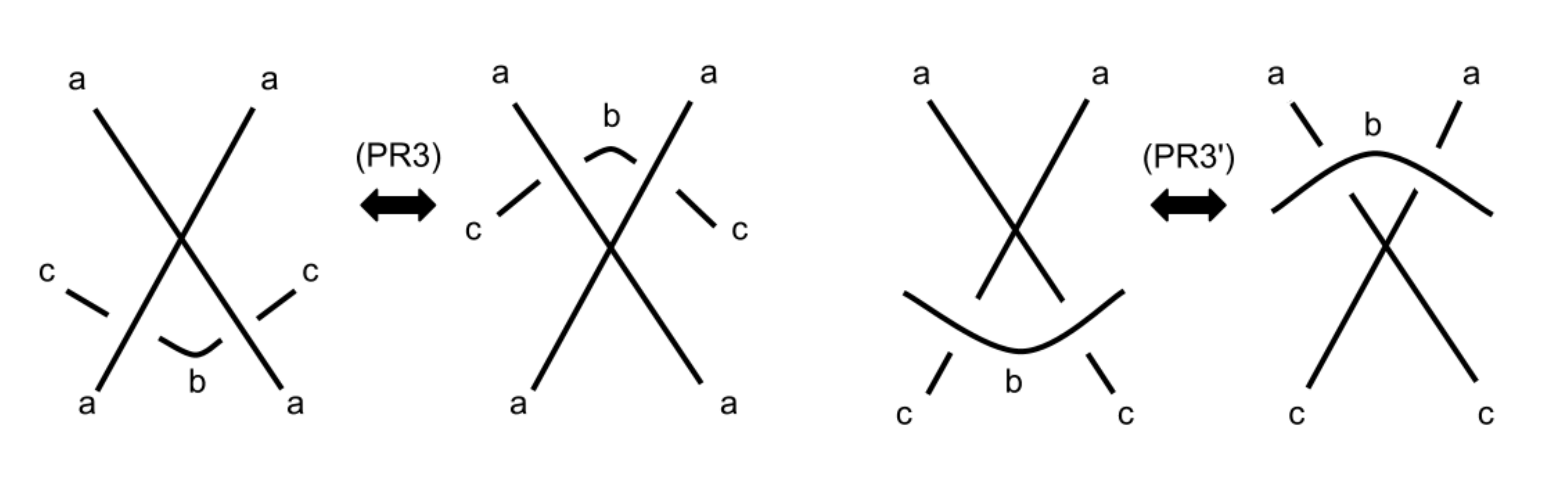}} \vspace*{8pt}
\caption{Colorings of pseudo-Reidemeister moves
\label{PRfig}}
\end{figure}


Consider the PR$3'$ move pictured in Figure~\ref{PRfig}. On the left, the top two arcs are forced by the precrossing to be colored by the same element, $a$. We color the over-arc with element $b$, which may or may not equal $a$. Then both remaining arcs are colored with $c=2b-a (mod\,p)$. Since both of the remaining arcs must be colored with the same element, we see that the induced coloring on the local picture after the move has been performed (pictured on the right) includes a valid coloring of the precrossing.
\end{proof}

Let's look at an example. According to Definition~\ref{tri}, the pseudodiagram $2\,1,2\,1,-(i,1,1)$ will not be strong 3-colorable. This is because, in every coloring of diagram $2\,1,2\,1,-(1,1,1)$ mod 3, all three colors appear in every crossing, so there is no coloring mod 3 of this pseudodiagram where only one color appears in a neighborhood of the precrossing. However, if we consider the weaker notion of $p$-colorability in Definition~\ref{modp}, our example will be colorable mod 3. Indeed, both of its resolutions, $2\,1,2\,1,-(1,1,1)$ and $2\,1,2\,1,-(-1,1,1)$, are colorable mod 3 in the classical sense (Fig. \ref{f00}). The existence of examples of this sort, prompts us to explore further the notion of colorability mod $p$ for pseudoknots. Let us recall our definition.

\bigskip

\begin{definition}\label{modp}
A pseudoknot is {\em $p$-colorable} (or {\em colorable mod} $p$) if all of its resolutions are colorable mod $p$.
\end{definition}

\begin{figure}[th]
\centerline{\includegraphics[width=4in]{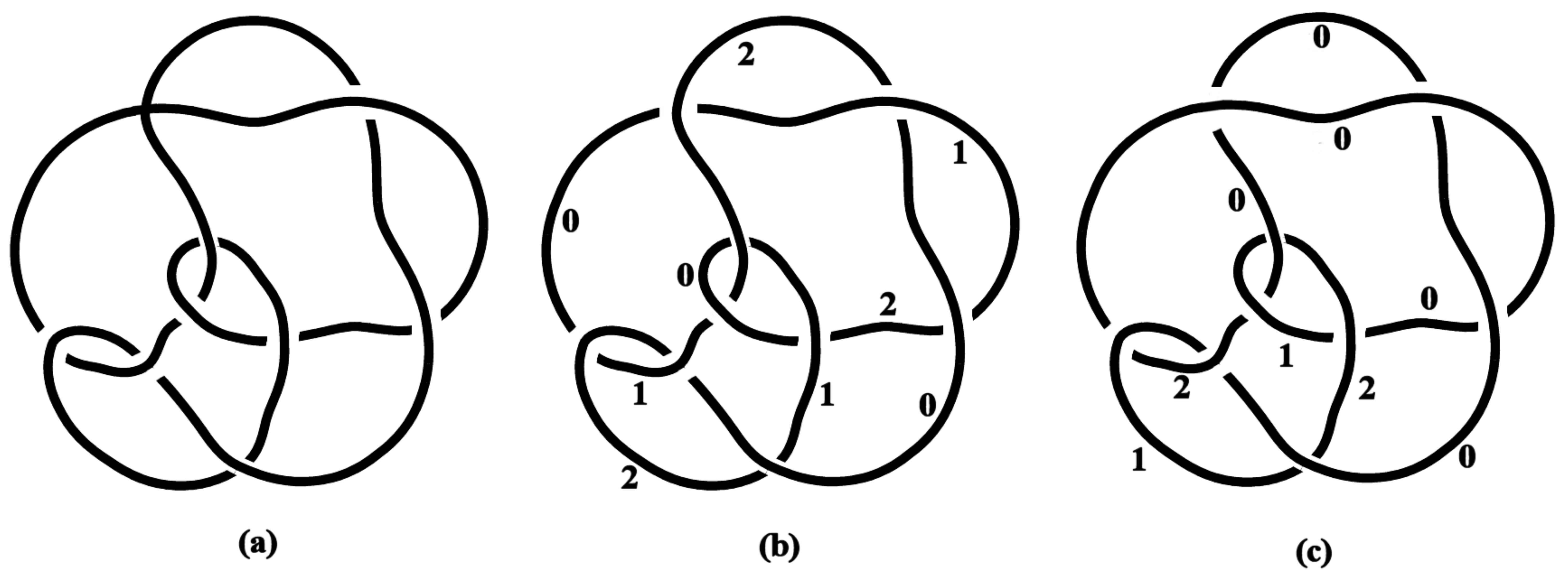}} \vspace*{8pt}
\caption{(a) Pseudodiagram $2\,1,2\,1,-(i,1,1)$; (b) resolution $2\,1,2\,1,-(1,1,1)$ and its coloring mod 3; (c) resolution $2\,1,2\,1,-(-1,1,1)$ and its coloring mod 3.
\label{f00}}
\end{figure}

\begin{theorem}Colorability mod $p$ is an invariant of pseudoknots.\end{theorem}

\begin{proof} Suppose that a pseudodiagram $P$ is colorable mod $p$, and let $D_P$ be one of its resolutions. (So $D_P$ is a knot diagram that is colorable mod $p$.) Suppose a (classical) Reidemeister move is applied to $P$ to produce a new diagram $P'$, and apply the corresponding Reidemeister move to $D_P$ to obtain the diagram $D_P'$. Note that $D_P'$ is a resolution of $P'$, and since mod $p$ colorability  is an invariant of classical knots, $D_P'$ remains mod $p$ colorable. Thus, $P'$ is mod $p$ colorable.

Suppose, instead, that a non-classical pseudo-Reidemeister (PR) move is applied to $P$ to produce $P''$. If the move applied is the PR1 move that introduces a new precrossing, then there are two new resolutions of $P''$ that are each related to $D_P$ by a Reidemeister 1 move. (One move introduces a positive kink in $D_P$ and one introduces a negative kink.) These two resolutions are colorable mod $p$ since $D_P$ is. If the move applied is the PR1 move that reduces the number of crossings by one, then a simple R1 move may be applied to $D_P$ to get the corresponding resolution of $P''$. Again, this resolution will be colorable mod $p$ since $D_P$ is. If a PR2, PR3, or PR3' move is applied to $P$, there is a unique corresponding Reidemeister 2 or 3 move that may be applied to $D_P$ to obtain a resolution, $D_P''$ of $P''$. (Note that the proof of this statement is similar to the proof of invariance of the WeRe-set in~\cite{ps}.) Thus $D_P''$ is also colorable mod $p$. Hence, $P''$ is colorable mod $p$.
\end{proof}

\begin{example}
Pseudoknots $3\,i\,3$ and $2\,1\,i\,1\,2$ (Fig. \ref{f02}) are $p$-colorable for $p=3$.
\end{example}

\begin{figure}[th]
\centerline{\includegraphics[width=1.4in]{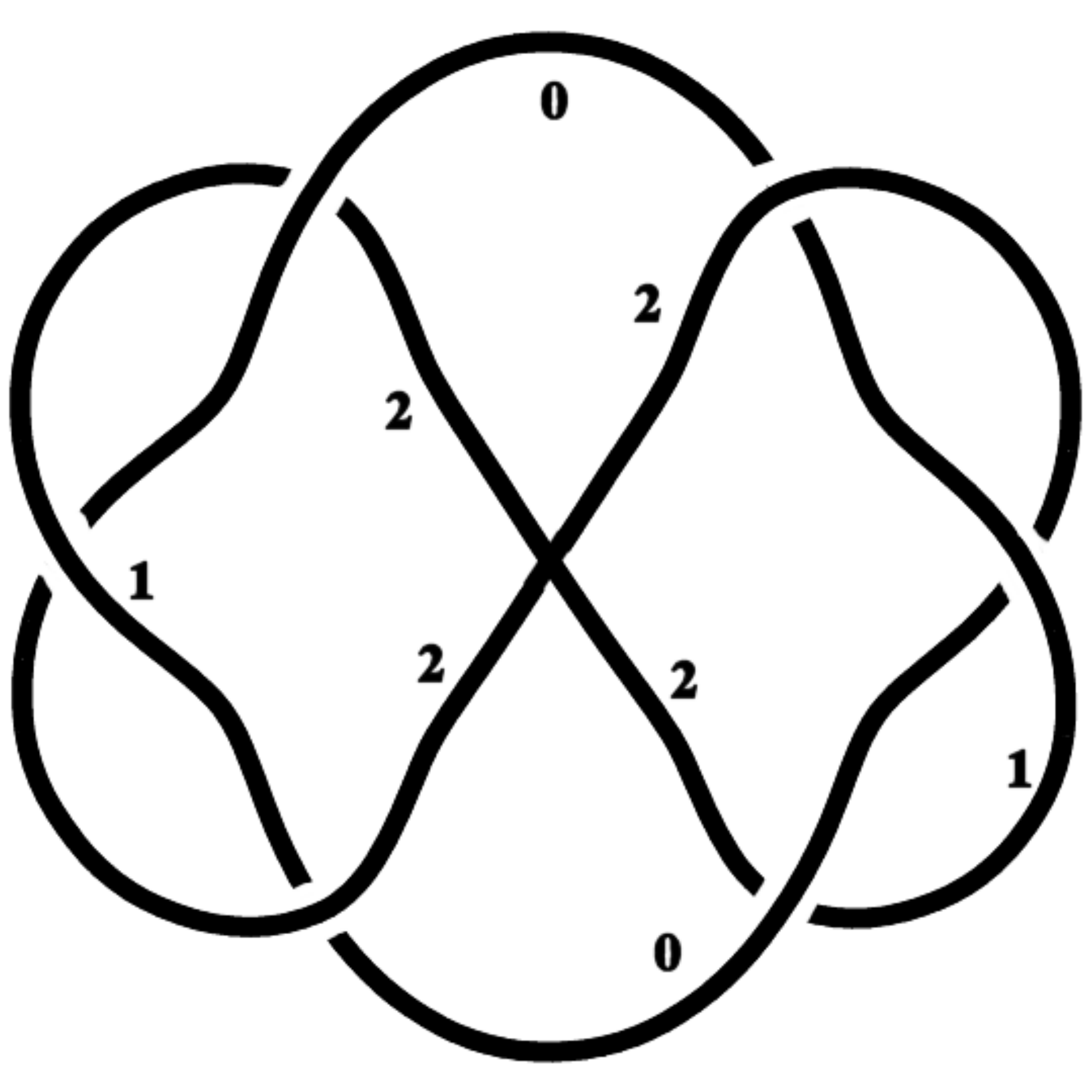}\hspace{.5in}\includegraphics[width=1.4in]{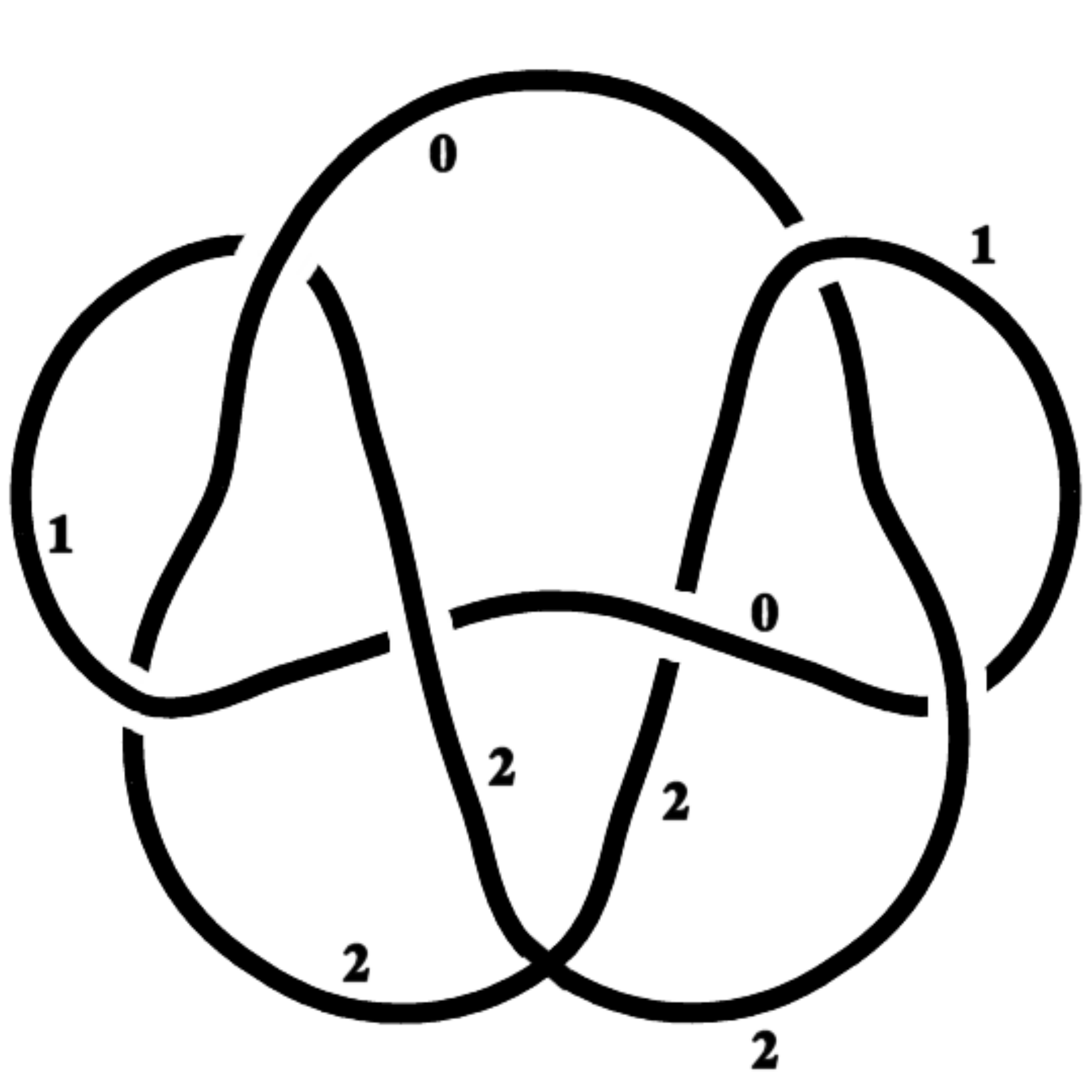}} \vspace*{8pt}
\caption{Pseudoknots $3\,i\,3$ (left) and $2\,1\,i\,1\,2$ (right) and their 3-colorings.
\label{f02}}
\end{figure}

Pseudoknots $6^*2.2\,0.i.1.1.1$ (Fig. \ref{f03}) and $6^*2.2\,0.1.1.1.i$ (Fig. \ref{f04}) are colorable mod 7 and mod 5, respectively.

\begin{figure}[th]
\centerline{\includegraphics[width=1.4in]{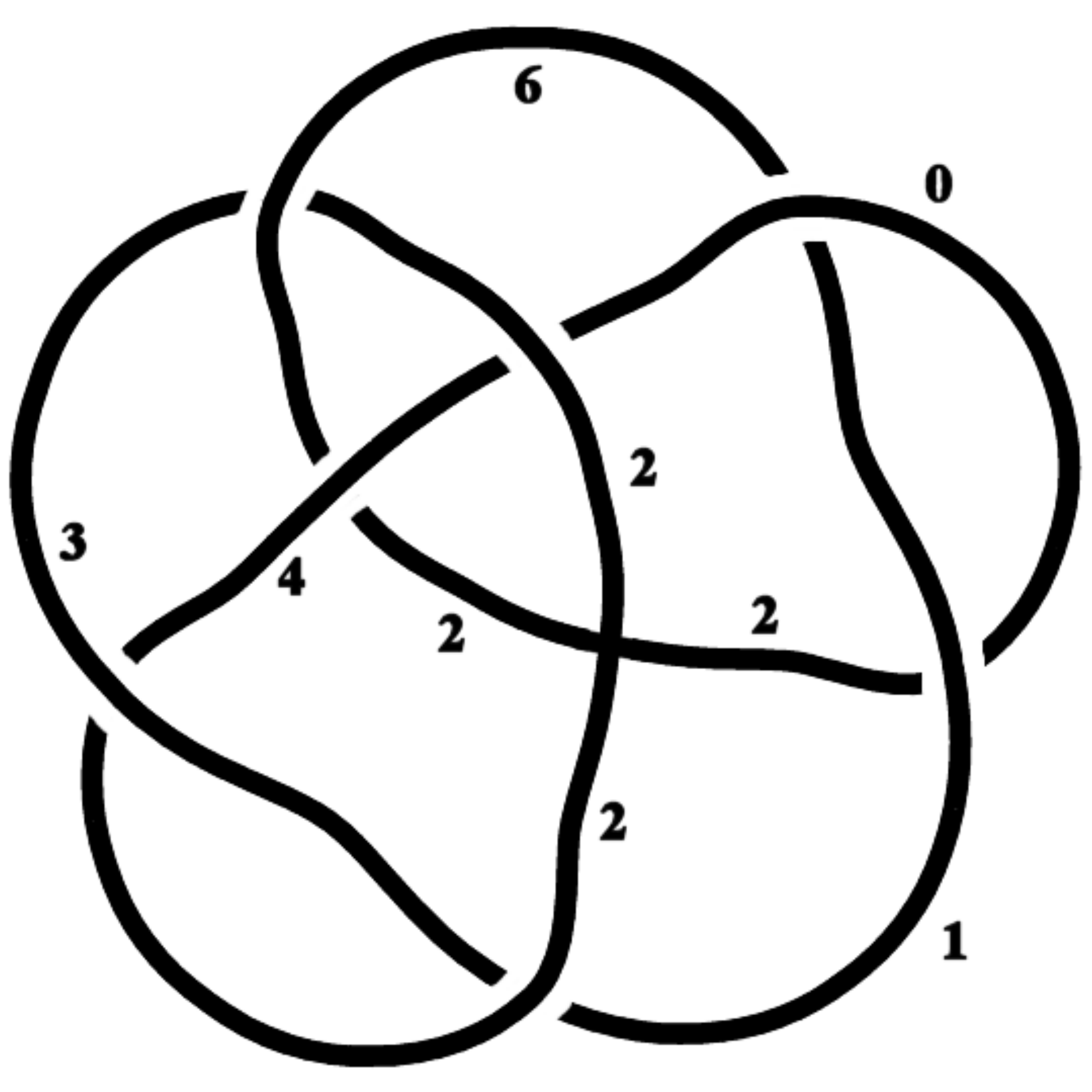}} \vspace*{8pt}
\caption{Pseudoknot $6^*2.2\,0.i.1.1.1$ and its coloring mod 7.
\label{f03}}
\end{figure}

\begin{figure}[th]
\centerline{\includegraphics[width=1.4in]{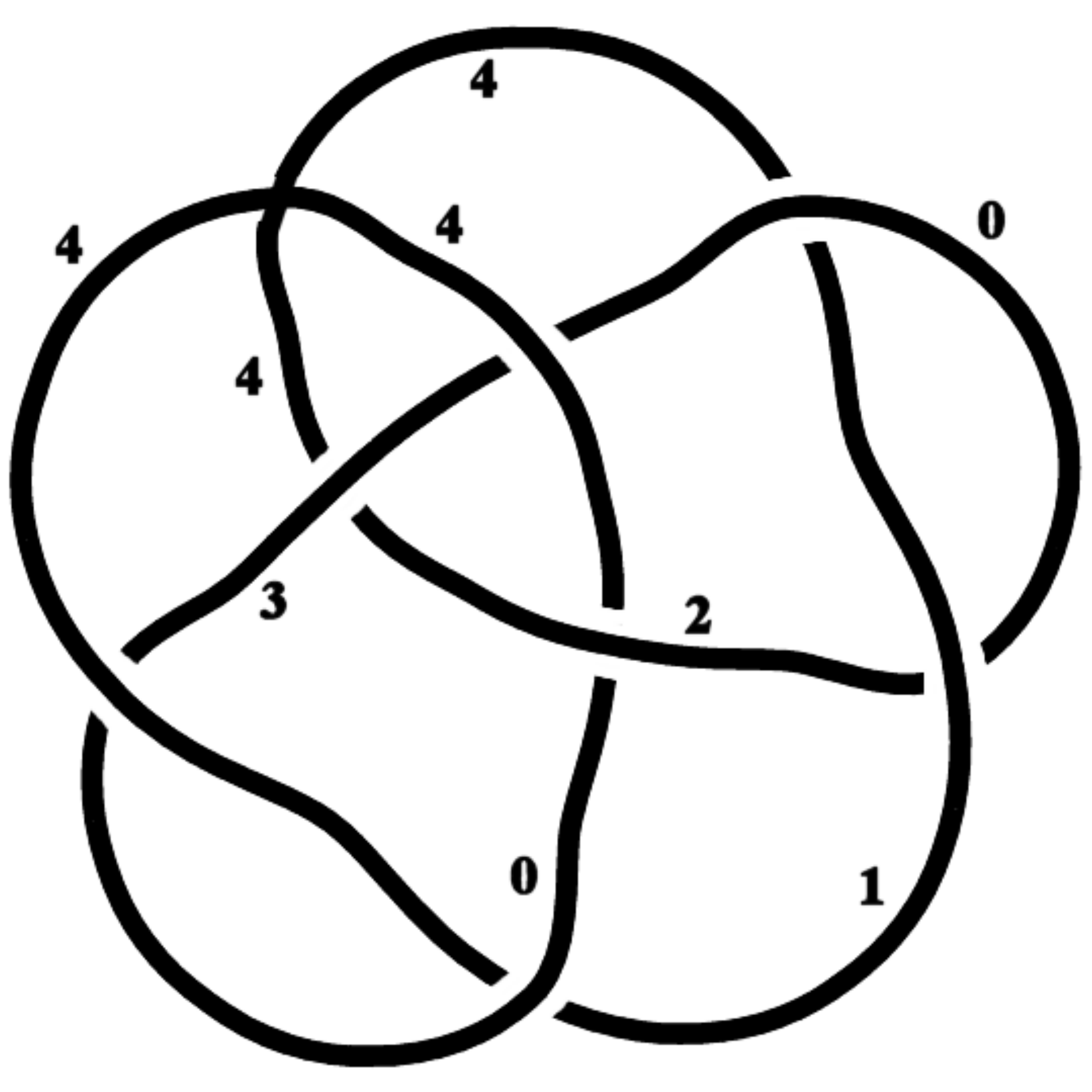}} \vspace*{8pt}
\caption{$6^*2.2\,0.1.1.1.i$ and its coloring mod 5.
\label{f04}}
\end{figure}

\begin{definition}
A pseudodiagram obtained from another pseudodiagram $K$ with $k$ precrossings ($k>1$) by resolving at most $k-1$ of its precrossings will be called a {\em pseudoresolution} of $K$.
\end{definition}

\begin{proposition} If a pseudodiagram $K$ containing $k>1$ precrossings is colorable mod $p$, then all of its pseudoresolutions are colorable mod $p$.\end{proposition}

\begin{example}
Pseudodiagram $8^*i.1.1.1.i.1.1.1$ and its pseudoresolutions

\noindent$8^*i.1.1.1.1.1.1.1$ and $8^*i.1.1.1.-1.1.1.1$ are colorable mod 3 (Fig. \ref{f05}).
\end{example}

\begin{figure}[th]
\centerline{\includegraphics[width=1.4in]{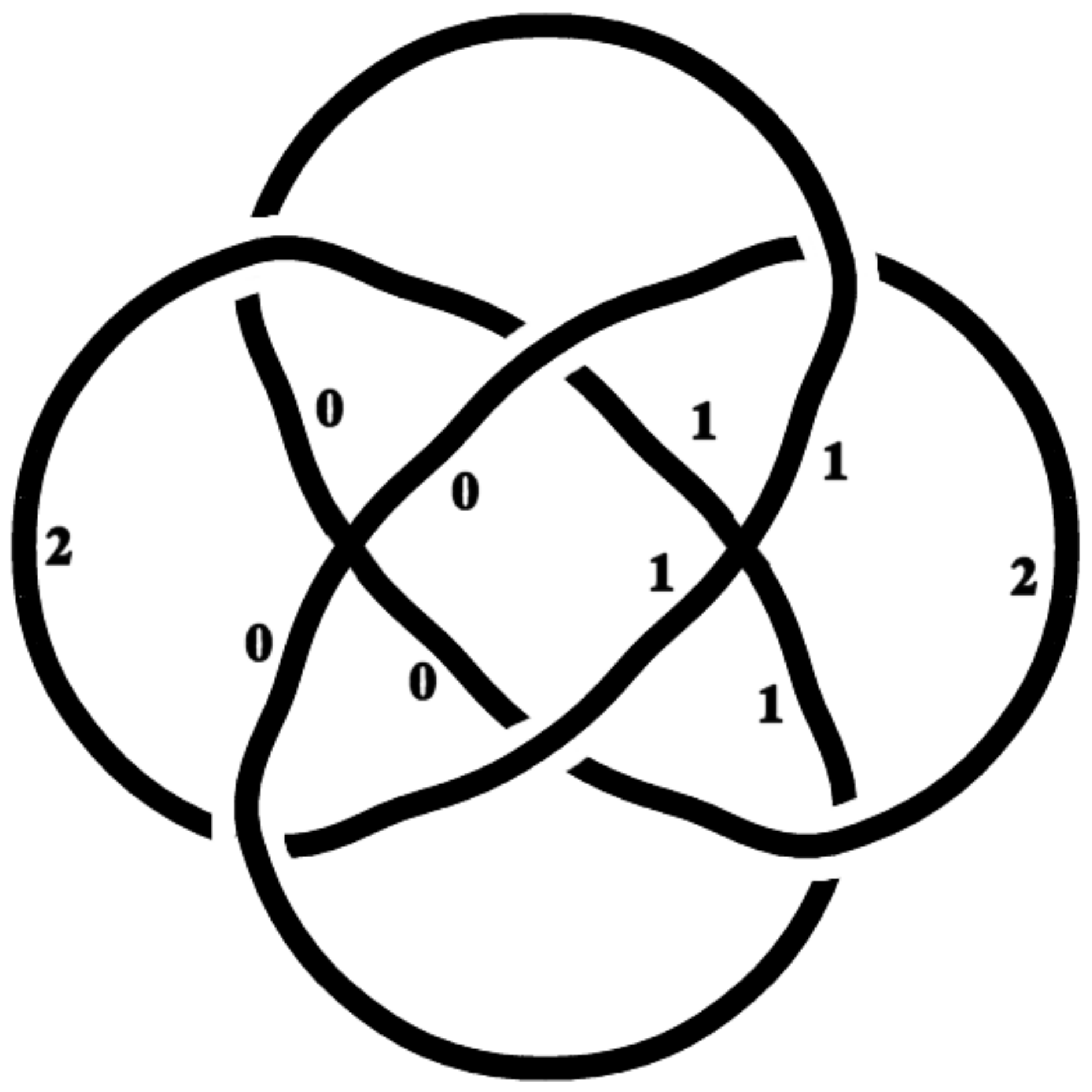}} \vspace*{8pt}
\caption{Pseudodiagram $8^*i.1.1.1.i.1.1.1$ and its coloring mod 3.
\label{f05}}
\end{figure}


\begin{figure}[th]
\centerline{\includegraphics[width=1.4in]{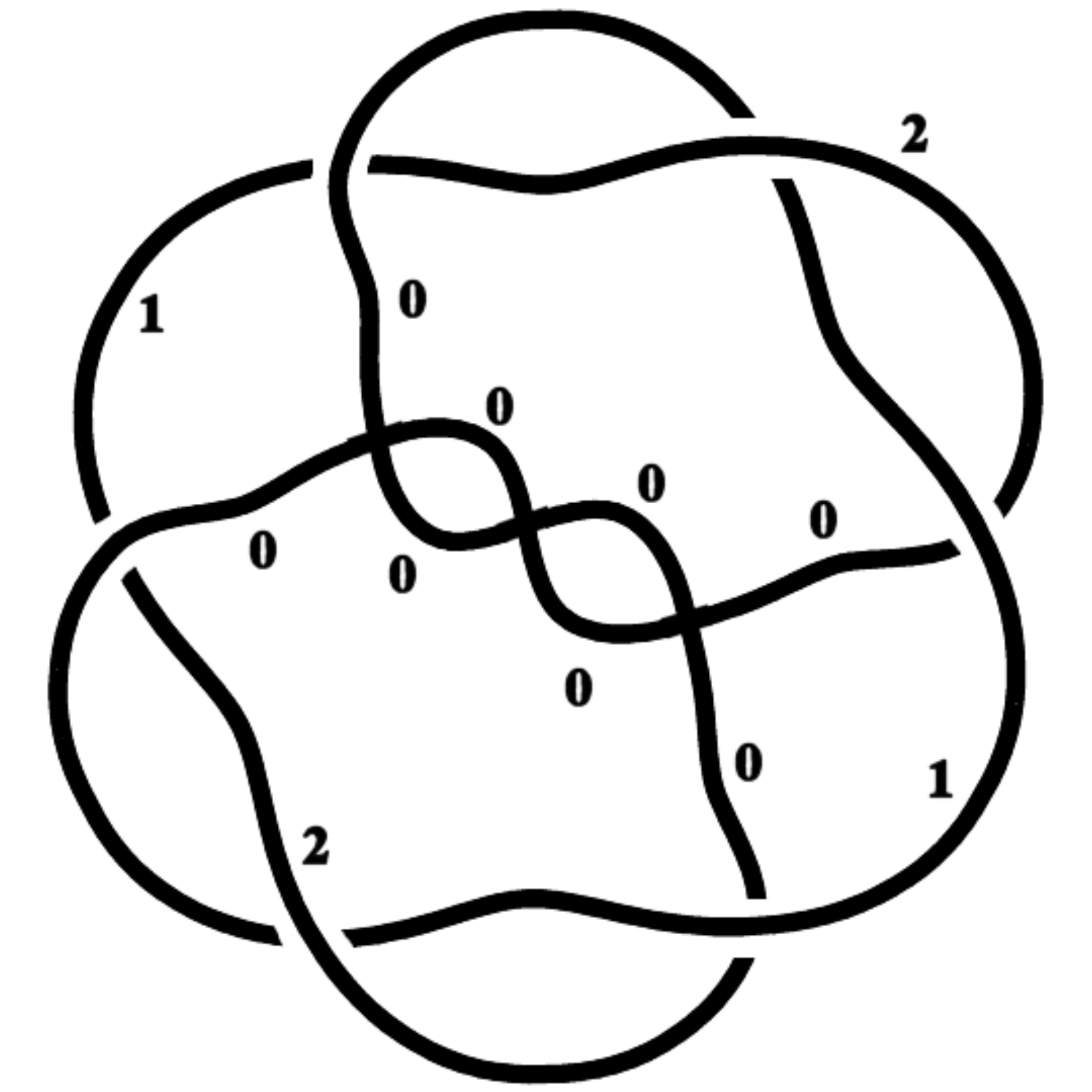}} \vspace*{8pt}
\caption{Pseudoknot $(i,i,i),3,-3$ and its coloring mod 3.
\label{f06}}
\end{figure}

For every pseudoknot, we can define its {\em coloring numbers} to be the set of numbers $p$ for which it is colorable mod $p$. For example, pseudoknot $9^*.i$ is colorable mod 3,  mod 5, and mod 15 (Fig. \ref{f07}). We see that the coloring numbers of a pseudoknot $K$ can be determined from $K$'s resolution set using the following notion.

\begin{figure}[th]
\centerline{\includegraphics[width=2.8in]{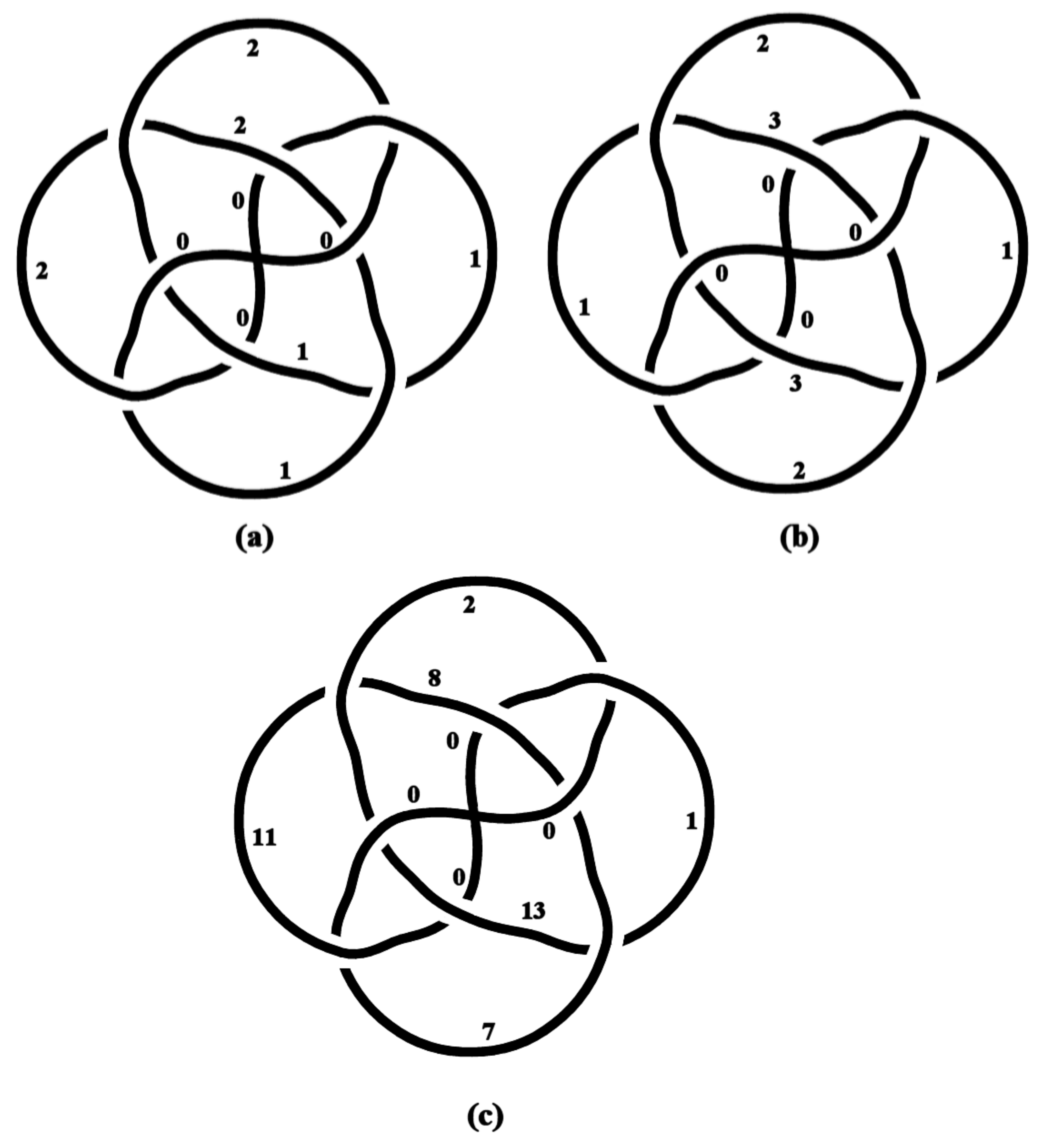}} \vspace*{8pt}
\caption{Pseudoknot $9^*.i$  and its coloring (a) mod 3; (b) mod 5; (c) mod 15.
\label{f07}}
\end{figure}

\begin{figure}[th]
\centerline{\includegraphics[width=3.8in]{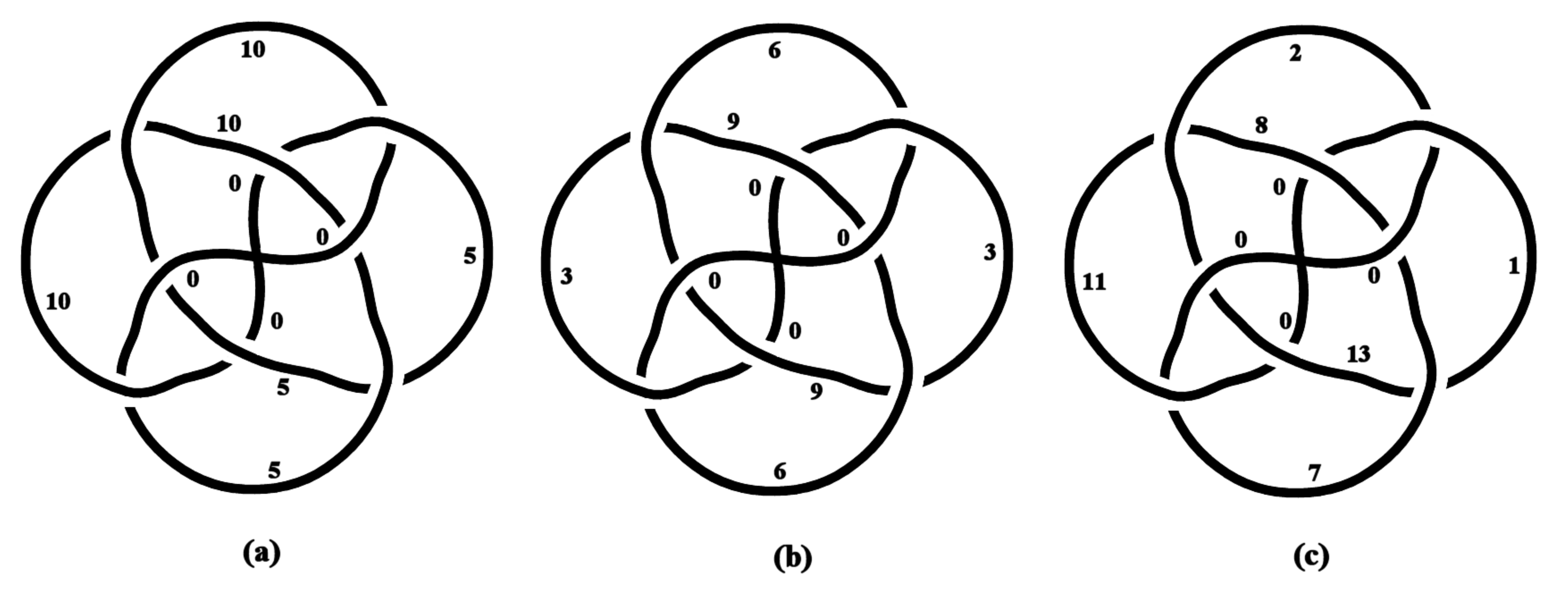}} \vspace*{8pt}
\caption{Pseudoknot $9^*.i$ with $d=15$ and its colorings for $p=15$ with (a) 3; (b) 4; (c) 7 colors.
\label{f08}}
\end{figure}

\begin{definition} Let $K$ be a pseudoknot with resolution set $\{$$K_1$, $\ldots $, $K_n$$\}$. Then $$d=GCD(Det(K_1),\ldots ,Det(K_n))$$ is called the {\em pseudodeterminant} of $K$.
\end{definition}

Note that, for classical knots, the notion of the pseudodeterminant coincides with the definition of the determinant of a knot. Also from this definition, we immediately have the following proposition, as illustrated in Fig. \ref{f08}.

\begin{proposition} If $K_1$, $\ldots $, $K_n$ are all resolutions of $K$, then $K$ will be colorable mod $p$ for every $p$ which divides its pseudodeterminant.\end{proposition}

We can use the notion of a pseudodeterminant to learn about the colorability of families of knots. Many families of knots given in Conway notation can be described using {\em pseudotwists}.

\begin{definition} A {\em pseudotwist} is a twist of length $n\ge 1$ which contains at least one precrossing.\end{definition}

Note that if a pseudotwist of length $n$ in a link pseudodiagram is replaced by a pseudotwist with length $n+2m$ where $m$ is an integer, then the number of components of the pseudolink is preserved. Performing the operation of replacing twists, we obtain families of pseudoknots and pseudolinks. In every family, a pseudotwist $i^n$ can be replaced by pseudotwist $(i^{n-l},(\pm 1)^l)$, ($n>l\ge 1$) in order to obtain all of its pseudoresolutions.


\bigskip

\begin{theorem}\label{Ptwist}
Suppose $L$ is the diagram of a pseudolink that contains at least one pseudotwist. Then any link pseudodiagram obtained from $L$ by replacing a pseudotwist in $L$ with a pseudotwist that has a length of the same parity will have the same pseudodeterminant.
\end{theorem}

Before we prove Theroem~\ref{Ptwist}, let us make some general remarks about the behavior of knot determinants. Suppose we have a link given in Conway notation that contains a twist of length $n$, where $n$ is an integer. Let's call this link $L(n)$. Now suppose we replace the length $n$ twist with a length $n+2m$ twist for some integer $m$ to obtain $L(n+2m)$. If we let $d_m=det(L(n+2m))$, then there exists an $r$ such that for all $m$, $d_m-d_{m-1}=r$. See~\cite{2} for more details. We make use of this result in the following proof.

\begin{proof}
Let $L$ be the diagram of a pseudolink that contains only one simple pseudotwist $(i)$. Denote the determinants of the resolutions of this precrossing by $a$ and $b$ ($a\le b$), and let $d=GCD(a,b)$ denote the pseudodeterminant of $L$. Let $L'$ be the pseudodiagram obtained from $L$ by replacing $(i)$ by another pseudotwist of odd length. In order to prove that the pseudodeterminant of $L$ equals that of $L'$, we first show as a base case that the pseudodeterminant remains unchanged if we replace $(i)$ by a pseudotwist of length three.

Notice that the pseudodeterminant remains unchanged by any permutation of crossings in a pseudotwist (since any permutation of crossings is equivalent to the original ordering by some sequence of PR2 moves). If we replace the pseudotwist $(i)$ in $L$ by $(i,1,-1)$, it is clear that pseudodeterminant remains unchanged, since the two tangles are related by a R2 move. If, on the other hand, we replace $(i)$ by $(i,1,1)$, $(i,-1,-1)$, $(i,i,1)$, $(i,i,-1)$, or $(i,i,i)$, the pseudodeterminants of the resulting pseudolinks will be $GCD(b,2b\pm a)=d$, $GCD(a,b\pm 2a)=d$, $GCD(a,b,b,2b\pm a)=d$, $GCD(b,a,a,b\pm 2a)=d$, or $GCD(a,a,a,b,b,b,b\pm 2a,2b\pm a)=d$, respectively.

To see why, consider the last example where $(i)$ is replaced by $(i,i,i)$. We claim that the determinant of this pseudoknot is $GCD(a,a,a,b,b,b,b\pm 2a,2b\pm a)=d$. If two precrossings in $(i,i,i)$ are resolved to be positive and one negative, the resulting link contains the twist $(1)$ in its Conway notation and there are three ways to choose where the negative crossing occurs. Similarly, if two precrossings in $(i,i,i)$ are resolved to be negative and one positive, the resulting link contains the twist $(-1)$ in its Conway notation and there are three ways to choose where the positive crossing occurs. This accounts for the terms $a$ and $b$ each repeated three times. The determinants $b\pm 2a$ and $2b\pm a$ correspond to a choice of all positive or all negative crossings.

In each case of replacing $(i)$ by a length three pseudotwist, the pseudodeterminant remains unchanged. The induction step is proven similarly since $(i)$ is a subtwist of any pseudotwist. Therefore, $(i)$ in $L$ may be replaced by any odd-length pseudotwist to produce a pseudolink $L'$ with the same pseudodeterminant as $L$.

The inductive proof for the even length pseudotwist is similar. Indeed, since $(i)$ is a subtwist of every pseudotwist, any pseudotwist (regardless of the parity of its length) can be replaced by any other pseudotwist with length of the same parity without changing the pseudodeterminant.
\end{proof}

For example, consider the family of pseudoknots comprised of $8^*(i^{2k})\,0::(i^{2m+1}).(-1).(-1).(-1)$ together with all of the pseudoknots obtained by replacing the pseudotwists $i^{2k}$ and $i^{2m+1}$ each by arbitrary pseudotwists of the same parity. Each pseudoknot in this family has pseudodeterminant 9 since, for instance,  $Det(8^*(i^{2})\,0::(i).(-1).(-1).(-1))=9$, by Theorem~\ref{Ptwist}.

\begin{remark} Computing the coloring numbers of pseudoknots with $n\le 9$ crossings, we find evidence that a very small portion of pseudoknots will be non-trivially colorable: from 8583 knots with $n\le 9$ crossings, only 112 will be non-trivially colorable, 70 with $d=3$, 11 with $d=5$, 5 with $d=7$, 23 with $d=9$, 1 with $d=11$, 1 with $d=15$, and 1 with $d=25$.\end{remark}

\section{Families of Pseudoknots and Conjectures}\label{example_conj}

\subsection{Pseudoknot families and their colorings}
Instead of considering particular pseudoknots and their colorings, we may explore the extension of colorability to families of pseudoknots (given in Conway notation). Take, for example, the pseudoknot $3\,i\,3$ with $d=3$. This pseudoknot is the first member of the family of pseudoknots $K=(2p+1)\,(i^{2k-1})\,(2q+1)$ which contains nontrivially colorable pseudoknots. All resolutions of $K$ are rational knots of the form $a\,b\,c$, with $a=2p+1$, $c=2q+1$, and $b\in \{-(2k-1),\ldots ,2k-1\}$, with determinant
 \begin{eqnarray*}
 D&=&a+c+abc\\
 &=&(2p+1)+(2q+1)+(2p+1)(2k-1)(2q+1)\\
 &=&2k(2p+1)(2q+1)-(4pq-1)
  \end{eqnarray*}
   for every $k$. (See~\cite{2} for more on computing determinants of rational knots.) The $GCD$ of these determinants, i.e., the pseudodeterminant, will be $d=GCD((2p+1)(2q+1),4pq-1)$. For example, pseudoknots $45\,(i^{2k-1})\,9$ have $d=27$, while pseudoknots $495\,(i^{2k-1})\,99$ have $d=297$, {\em etc.}

In Appendix A, we provide a table of families of colorable pseudoknots that are derived from pseudoknots with at most $n=9$ crossings. In the same way as for the example above (which is entry (1) in our table), we can obtain general formulae for the pseudodeterminant $d$ of pseudoknot families by using general formulae for the determinants of classical knots. For example, for family (2) from the table, the determinant is
\begin{eqnarray*}
D&=&(2p+1)-(2q+1)-(2p+1)(2k-1)(2q+1)\\
&=&2k(2p+1)(2q+1)+(4pq+4p+1)
\end{eqnarray*}
 for every $k$, so the pseudodeterminant is $d=GCD((2p+1)(2q+1),4pq+4p+1)$. Knowing that the determinant $D$ of classical rational knots and links $a\,1\,b\,1\,c$ is given by the general formula
 \begin{eqnarray*}
 D&=&abc+ab+2ac+bc+a+b+c\\
&=&(2p)(2k-1)(2q)+(2p)(2k-1)+2(2p)(2q)\\
&&+(2k-1)(2q)+(2p)+(2k-1)+(2q)\\
&=&2k(2p+1)(2q+1)+(4pq-1)
\end{eqnarray*}
 for every $k$, we also conclude that for family (3), the pseudodeterminant is again $d=GCD((2p+1)(2q+1),4pq-1)$. Proceeding in this way, we determine all entries in the table.

\subsection{The Kauffman-Harary Conjecture}


A knot diagram $D$ is said to have the {\em Kauffman-Harary property} (or {\em KH property} for short) if it can be endowed with a nontrivial coloring mod $p$ which assigns a unique color to each arc~\cite{KH}. A pseudodiagram $D$ is said to have the KH property if all of its resolutions have the KH property. The smallest example of a pseudodiagram with the KH property is the pseudodiagram $(3)\,(i)\,(-3)$. Both of its resolution diagrams $(3)\,(1)\,(-3)$ and $(3)\,(-1)\,(-3)$ have the KH property for the coloring mod 9 (Fig. \ref{f09}). The KH property is not just a property of particular pseudoknots, but of the entire families. For example, all pseudoknots $(2p+1)\,(i)\,-(2p-1)$ with $d=(2p+1)^2$ have KH property. Indeed, we find that both resolution diagrams $(2p+1)\,(1)\,-(2p+1)$ and $(2p+1)\,(-1)\,-(2p+1)$ are colorable mod $d$ with $4p+3$ different colors. Figure \ref{f10} shows the second member of this family, pseudoknot $5\,i\,-5$ with $d=25$ and its resolution diagrams $(5)\,(1)\,(-5)$ and $(5)\,(-1)\,(-5)$ colored mod 25 with 11 colors.

\begin{figure}[th]
\centerline{\includegraphics[width=4.2in]{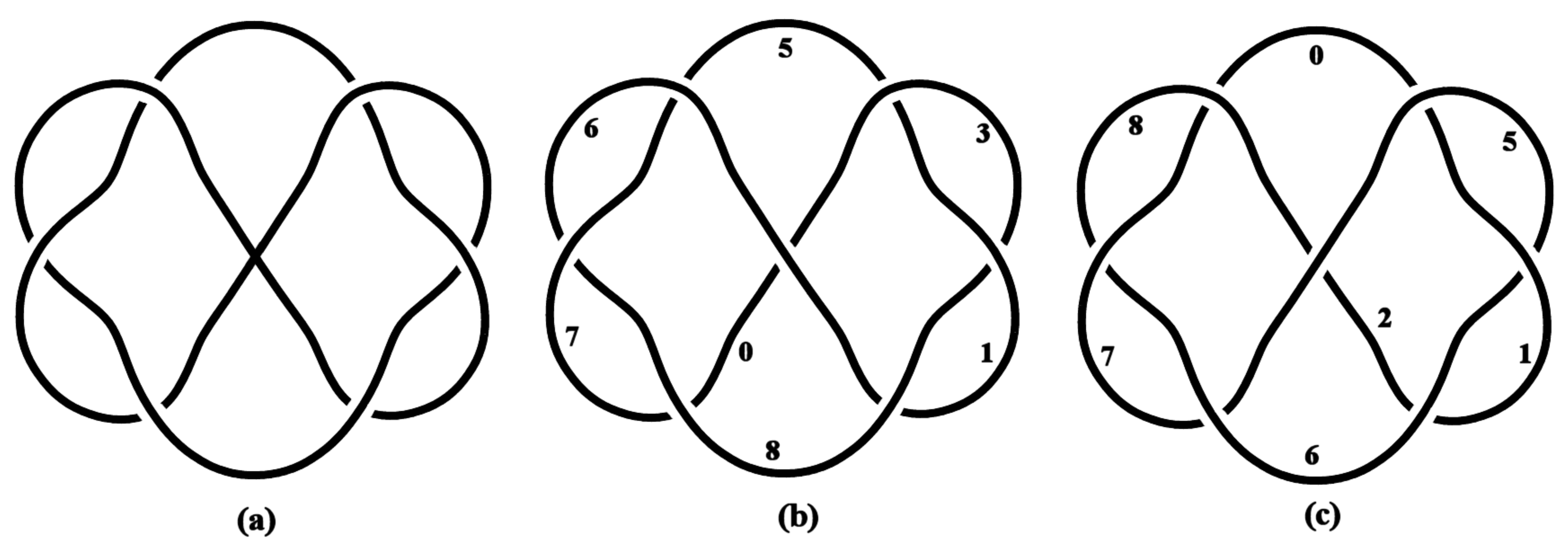}} \vspace*{8pt}
\caption{(a) Pseudoknot $(3)\,(i)\,(-3)$ with $d=9$ and colorings of its resolutions (b) $(3)\,(1)\,(-3)$ and (c) $(3)\,(-1)\,(-3)$ colored with 7 colors.
\label{f09}}
\end{figure}

\begin{figure}[th]
\centerline{\includegraphics[width=4.2in]{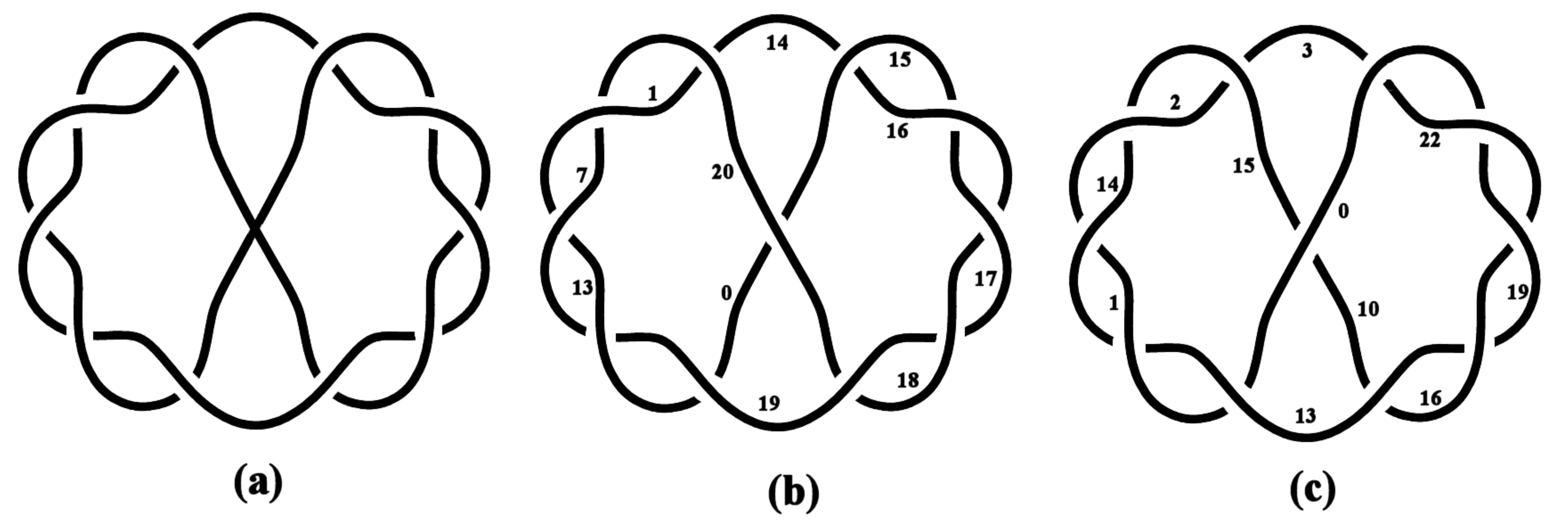}} \vspace*{8pt}
\caption{(a) Pseudoknot $(5)\,(i)\,(-5)$ with $d=25$ and colorings of its resolutions (b) $(5)\,(1)\,(-5)$ and (c) $(5)\,(-1)\,(-5)$ colored with 11 colors.
\label{f10}}
\end{figure}

\begin{figure}[th]
\centerline{\includegraphics[width=4.2in]{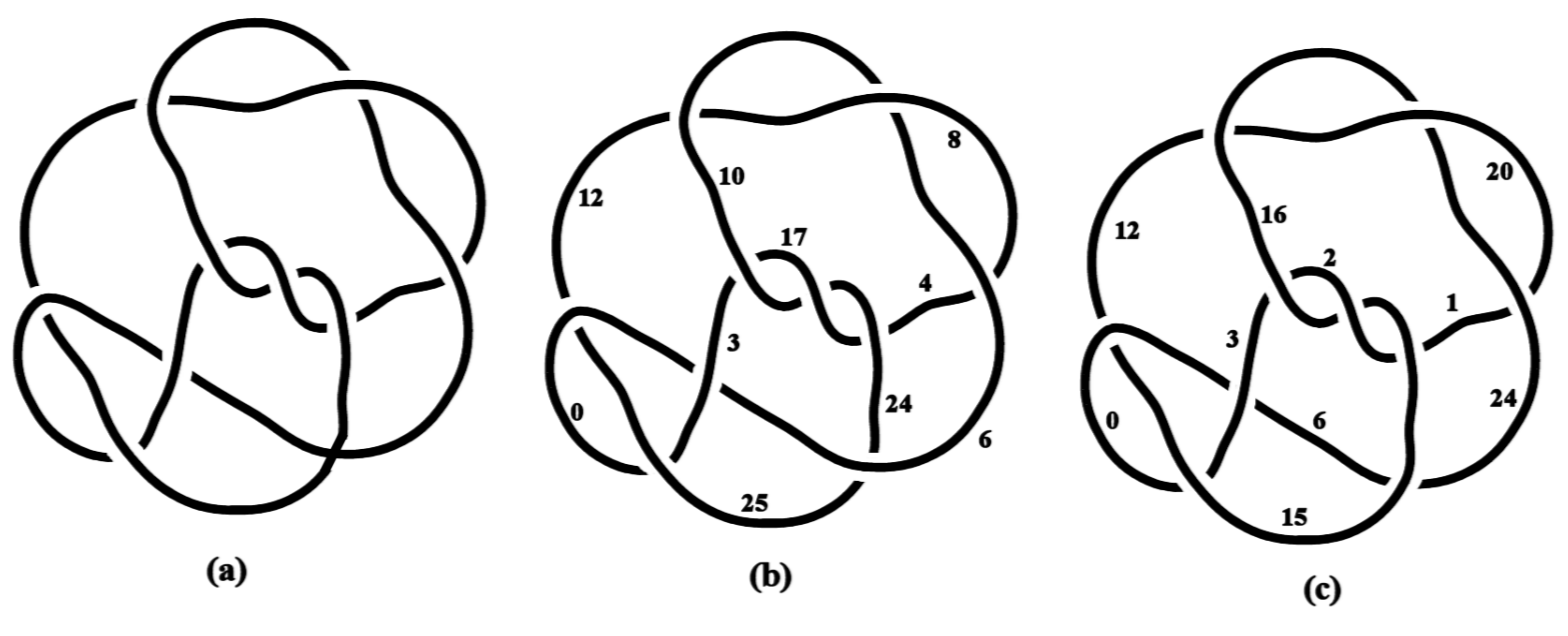}} \vspace*{8pt}
\caption{(a) Pseudoknot $2\,1\,i,3,-3$ with $d=27$ and colorings mod 27 of its resolutions (b) $2\,1\,1,3,-3$ and (c) $2\,1\,(-1),3,-3$ colored with 10 colors.
\label{f11}}
\end{figure}

\begin{figure}[th]
\centerline{\includegraphics[width=4.2in]{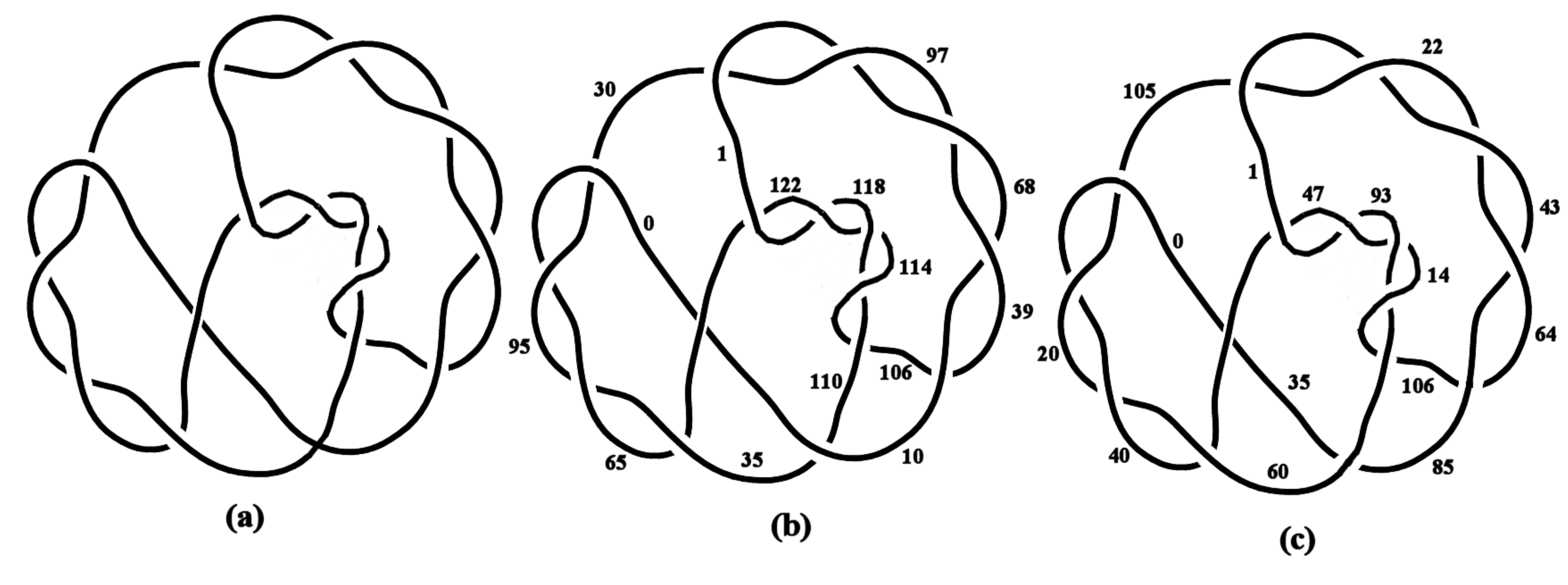}} \vspace*{8pt}
\caption{(a) Pseudoknot $4\,1\,i,5,-5$ with $d=125$ and colorings mod 125 of its resolutions (b) $4\,1\,1,5,-5$ and (c) $4\,1\,(-1),5,-5$ colored with 16 colors.
\label{f12}}
\end{figure}

Another example of a pseudoknot family with the KH property is the family $(2p)\,1\,i,3,-3$ with $d=18p+9$. Both of its resolution diagrams $(2p)\,1\,1,3,-3$ and $(2p)\,1\,(-1),3,-3$ are colorable mod $d$ with $2p+8$ colors (Fig. \ref{f11}). This family can be extended to the pseudoknot family $(2p)\,1\,i,(2p+1),-(2p+1)$ with the KH property, where $d=(2p+1)^3$, and each diagram is colorable mod $d$ with $6p+4$ colors (Fig. \ref{f12}).

\bigskip

An interesting question involving the KH property for pseudoknots relates to the idea of a pseudoalternating pseudodiagram. A pseudodiagram is {\em pseudoalternating} if it has an alternating resolution.

\begin{conjecture}
A pseudodiagram with the KH property cannot be pseudoalternating.
\end{conjecture}

We conclude our exploration of pseudoknot colorability with another idea concerning a possible direction for future work. More open problems about colorings of pseudoknots relate to the number of colors used for coloring mod $p$. For every pseudoknot colorable mod $p$, we can define $mincol(K,p)$ and $maxcol(K,p)$ as the minimal and maximal number of colors needed for coloring $K$, where both numbers are taken over all diagrams of $K$ \cite{maxcol}. Just as in the case of classical knots, the first invariant will be very hard to compute. This is because, for every knot with $m=mincol(K,p)$ and $M=maxcol(K,p)$, there exists a diagram that is colorable using $k$ colors, where $m\le k\le M$. In other words, the coloring spectrum of every knot is the complete set $m\le k\le M$  \cite{maxcol}. The same holds for pseudoknots. For a pseudoknot $K$, $m=mincol(K,p)$ will be the maximum of minimal coloring numbers of its resolutions, and $M=maxcol(K,p)$ is equal to pseudodeterminant $d$ of $K$.


\appendix
\section{Pseudoknot Tables}

In this section, we provide a table of pseudoknot families together with their pseudodeterminants.

\begin{landscape}
\footnotesize
\noindent  \begin{tabular}{|c|c|c|} \hline
 & $K$ & $d$ \\ \hline
1) &$(2p+1)\,(i^{2k-1})\,(2q+1)$ & $GCD((2p+1)(2q+1),4pq-1)$  \\ \hline
2) &$(2p+1)\,(i^{2k-1})\,-(2q+1)$ &$GCD((2p+1)(2q+1),4pq+4p+1)$  \\ \hline
3) &$(2p)\,1\,(i^{2k-1})\,1\,(2q)$ & $GCD((2p+1)(2q+1),4pq-1)$  \\ \hline
4) &$(2p+1),(2q+1),(i^{2k})$  & $GCD((2p+1)(2q+1),p+q+1)$ \\ \hline
5) &$(2p+1),-(2q+1),(i^{2k})$  & $GCD((2p+1)(2q+1),p-q)$  \\ \hline
6) &$(2p+1),(2q)\,1,(i^{2k})$  & $GCD((2p+1)(2q+1),4pq+4q+1)$  \\ \hline
7) &$(2p+1),-(2q)\,(-1),(i^{2k})$ & $GCD((2p+1)(2q+1),4pq-1)$  \\ \hline
8) &$(2p+1)\,(i^{2k})\,1\,(2q)$ & $GCD((2p+1)(2q+1),4pq+4q+1)$  \\ \hline
9) &$(2p+1)\,(i^{2k})\,(-1)\,(-2q)$ & $GCD((2p+1)(2q+1),4pq-1)$  \\ \hline
10) &$(2p)\,1,(2q)\,1,(i^{2k})$  & $GCD((2p+1)(2q+1),4pq+p+q)$ \\ \hline
11) &$6^*(2p).(2q)\,0.(i^{2k-1})$ & $GCD(12pq-2p-2q-1,3p+3q+1)$  \\ \hline
12) &$6^*(2p).(2q)\,0.(i^{2k-1})\,0$ & $GCD(12pq-2p-2q-1,12pq+4p+4q+1)$  \\ \hline
13) &$6^*(2p).(2q)\,0.(i^{2k-1}).(-1).(-1).(-1)$  & $GCD(12pq-10p-10q+3,3p+3q-1)$ \\ \hline
14) &$6^*(2p).(2q)\,0.(i^{2k-1})\,0.(-1).(-1).(-1)$ & $GCD(12pq-10p-10q+3,12pq-4p-4q+1)$  \\ \hline
15) &$6^*(2p).(2q)\,0::(i^{2k-1})$  & $GCD(4pq+2p+2q-3,4pq+3p+3q)$ \\ \hline
16) &$6^*(2p).(2q)\,0::(i^{2k-1})\,0$  & $GCD(4pq+2p+2q-3,4pq+4p+4q+3)$  \\ \hline
17) &$8^*(i^{2k-1})::(i^{2m-1})$ & 3  \\ \hline
18) &$8^*(i^{2k-1})\,0::(i^{2m-1})$ & 3  \\ \hline
19) &$8^*(i^{2k-1})\,0::(i^{2m-1})\,0$  & 3  \\ \hline
20) &$(2p+1),(2q+1),(i^{2k})+(i^{2m-1})$ & $GCD((2p+1)(2q+1),4pq+4p+4q+3)$  \\ \hline
21) &$(2p+1),-(2q+1),(i^{2k})+(i^{2m-1})$  & $GCD((2p+1)(2q+1),4pq+4q+1)$  \\ \hline
22) &$(2p)\,(2q)\,(i^{2k-1})\,(2r)\,(2s)$  &  $GCD(16pqrs-8pqs-8prs+4pq+4rs-2p-2s+1,16pqrs+4pq+4rs+1)$  \\  \hline
23) &$(2p)\,(2q)\,(i^{2k-1})\,-(2r)\,-(2s)$ & $GCD(16pqrs+8pqs-8prs+4pq+4rs-2p+2s+1,16pqrs+4pq+4rs+1)$  \\ \hline
24) &$(2p+1),(2q)\,1,(i^{2k})+(i^{2m-1})$  & $GCD((2p+1)(2q+1),4pq+p+3q+1)$ \\ \hline
25) &$(2p+1),-(2q)\,(-1),(i^{2k})+(i^{2m-1})$ & $GCD((2p+1)(2q+1),4pq+p+q)$  \\ \hline
26) &$(2p)\,1,(2q)\,1,(i^{2k})+(i^{2m-1})$ & $GCD((2p+1)(2q+1),12pq+4p+4q+1)$  \\ \hline
27) &$6^*(i^{2k-1}).(2p):(i^{2m}).(2q)\,0$ & $GCD(12pq+4p+4q+1,4pq+4p+4q+3)$  \\ \hline
28) &$6^*(i^{2k-1})\,0.(2p):(i^{2m}).(2q)\,0$ & $GCD(12pq+4p+4q+1,4pq+4p+4q+3)$  \\ \hline
29) &$6^*.(i^{2k}):-(2p).(2q)\,0$ & $GCD(8pq+6p-4q-1,2pq+2p-3q-1)$  \\ \hline
30) &$6^*.(2p).(i^{2k-1}).-(2q).(2r)\,0.(i^{2m-1})$ & $GCD(8pqr+8pq-4pr+4p-2q+1,4pqr+4pq-4pr+2qr+q-r)$  \\ \hline
\end{tabular}

\noindent  \begin{tabular}{|c|c|c|} \hline
31) &$6^*.(2p):(i^{2k}).(2q)\,0$ & $GCD(12pq+4p+4q+1,4pq+4p+4q+3)$  \\ \hline
32) &$(2p)\,1\,1\,(i^{2k-1})\,1\,1\,(2q)$ & $GCD(2p+2q+1,(4p+1)(4q+1))$  \\ \hline
33) &$8^*(i^{2k})\,0::(i^{2m-1})$  & 3 \\ \hline
34) &$8^*(i^{2k})\,0::(i^{2m-1})\,0$  & 3  \\ \hline
35) &$8^*(i^{2k})\,0::(i^{2m-1}).(-1).(-1).(-1)$ & 9  \\ \hline
36) &$8^*(i^{2k})\,0::(i^{2m-1})\,0.(-1).(-1).(-1)$ & 9  \\ \hline
37) &$8^*(2p)\,0.(-1).(i^{2k-1}).(-1).(-1).(-1).(i^{2m-1}).(-1)$ & $GCD(8p+1,9)$  \\ \hline
38) &$8^*(2p)\,0.(-1).(i^{2k-1})\,0.(-1).(-1).(-1).(i^{2m-1}).(-1)$ & $GCD(8p+1,9)$  \\ \hline
39) &$8^*(2p)\,0.(-1).(i^{2k-1})\,0.(-1).(-1).(-1).(i^{2m-1})\,0.(-1)$ & $GCD(8p+1,9)$  \\ \hline
40) &$(i^{2k-1}),(2p+1),(2q+1)$ & $GCD(4pq-1,p+q+1)$  \\ \hline
41) &$(i^{2k-1}),(2p)\,1,(2q)\,1$ & $GCD(4pq-1,4pq+p+q)$   \\ \hline
42) &$6^*(2p)\,0.(i^{2k})\,0:(2q).(i^{2m-1})$ & $GCD(4pq-1,8pq+3p+3q+1)$  \\ \hline
43) &$6^*(2p)\,0.(i^{2k})\,0:(2q).(i^{2m-1})\,0$ & $GCD(4pq-1,(2p+1)(2q+1))$  \\ \hline
44) &$6^*(2p).(i^{2k-1}).(2q):(2r)\,0$ & $GCD(16pqr+4pq-4pr-4qr-1,16pqr+4pq+4pr+4qr+2p+2q+1)$  \\ \hline
45) &$6^*(2p).(i^{2k-1})\,0.(2q):(2r)\,0$ & $GCD(16pqr+4pq-4pr-4qr-1,4pr+4qr+p+q+1)$  \\ \hline
46) &$6^*(2p):(2q):(i^{2k})\,0$  & $GCD(4pq+4p+4q+3,4pq+3p+3q)$ \\ \hline
47) &$9^*(i^{2k-1})::::(i^{2m-1})$ & 5  \\ \hline
48) &$9^*(i^{2k-1})\,0::::(i^{2m-1})$ & 5  \\ \hline
49) &$9^*(i^{2k-1})\,0::::(i^{2m-1})\,0$ & 5  \\ \hline
50) &$9^*.(i^{2k-1}):.(i^{2m-1}):.(i^{2n-1})$ & 3  \\ \hline
51) &$9^*.(i^{2k-1})\,0:.(i^{2m-1}):.(i^{2n-1})$  & 3 \\ \hline
52) &$9^*.(i^{2k-1})\,0:.(i^{2m-1})\,0:.(i^{2n-1})$ & 3  \\ \hline
53) &$9^*.(i^{2k-1})\,0:.(i^{2m-1})\,0:.(i^{2n-1})\,0$ & 3  \\ \hline
54) &$9^*.(i^{2k-1}).(-1):(i^{2m-1}).(-1):(i^{2n-1}).(-1)$ & 9  \\ \hline
55) &$9^*.(i^{2k-1})\,0.(-1):(i^{2m-1}).(-1):(i^{2n-1}).(-1)$ & 9  \\ \hline
56) &$9^*.(i^{2k-1})\,0.(-1):(i^{2m-1})\,0.(-1):(i^{2n-1}).(-1)$  & 9 \\ \hline
57) &$9^*.(i^{2k-1})\,0.(-1):(i^{2m-1})\,0.(-1):(i^{2n-1})\,0.(-1)$  & 9 \\ \hline
58) &$6^*(i^{2k})\,0:(2p)\,0:(2q)\,0$ & $GCD(12pq+4p+4q+1,3p+3q+1)$  \\ \hline
59) &$6^*(2p)\,0.(i^{2k-1}).(2q)\,0:(2r)\,0$  & $GCD(4pq+4pr+4qr-4r-1,4pq+4pr+4qr+2p+2q+4r+1)$ \\ \hline
60) &$6^*(2p)\,0.(i^{2k-1})\,0.(2q)\,0:(2r)\,0$  & $GCD(4pq+4pr+4qr-4r-1,4pq+4pr+4qr+p+q)$ \\ \hline
\end{tabular}
\end{landscape}

\end{document}